\documentclass[12pt,a4paper]{article}
\usepackage{amssymb}
\usepackage{amsmath, enumerate,amsthm,fullpage,wasysym,textcomp}
\usepackage{epsfig}
\numberwithin{equation}{section}
\theoremstyle{plain}
\newtheorem{thm}{Theorem}[section]
\newtheorem{prop}[thm]{Proposition}
\newtheorem*{mthm}{Main Theorem}
\newtheorem{cor}[thm]{Corollary}
\newtheorem{lem}[thm]{Lemma}
\theoremstyle{definition}
\newtheorem{exa}[thm]{Example}
\newtheorem{conj}[thm]{Conjecture}
\newtheorem{rem}[thm]{Remark}
\newtheorem{defi}[thm]{Definition}
\newtheorem{notation}[thm]{Notation}

\newcommand{\real}{\mathbb{R}}
\newcommand{\comp}{\mathbb{C}}
\DeclareMathOperator*{\nat}{\mathbb{N}}

\usepackage{calc}
\newcounter{PartitionDepth}
\newcounter{PartitionLength}
\newcommand{\lhalf}[2]{
 \begin{picture}(#2,#1)
 \put(-#2,-#1){\line(0,1){#1}}
 \put(-#2,-#1){\line(1,0){#2}} 
 \end{picture}}

\newcommand{\rhalf}[2]{
 \begin{picture}(#2,#1)
 \put(-#2,-#1){\line(0,1){#1}}
 \put(-#2,-#1){\line(-1,0){#2}} 
 \end{picture}}

\begin{document}
\title{ $k$-divisible random variables in free probability}
\author{Octavio Arizmendi\footnote{Supported by DFG-Deutsche Forschungsgemeinschaft Project SP419/8-1} \\ Universit\"{a}t des Saarlandes, FR $6.1-$Mathematik,\\ 66123 Saarbr\"{u}cken, Germany \\ }

%\date{\today}

\maketitle

\begin{abstract}
 We introduce and study the notion of $k$-divisible elements in a non-commutative probability space. A $k$-divisible element is a (non-commutative) random variable whose $n$-th moment vanishes whenever $n$ is not a multiple of $k$.
%In this note we study combinatorial, algebraic and analytic aspects of these $k$-divisible elements and their associated $k$-symmetric distributions in relation to free probability.

First, we consider the combinatorial convolution $*$ in the lattices $NC$ of non-crossing partitions and $NC^k$ of $k$-divisible non-crossing partitions and show that convolving $k$ times with the zeta-function in $NC$ is equivalent to convolving once with the zeta-function in $NC^k$. 
%This gives new ways of counting objects like $k$-equal partitions, $k$-divisible partitions and $k$-multichains both in $NC$ and $NC^k$.
%From the combinatorial results on the lattice of $k$-divisible non-crossing partitions, 
Furthermore, when $x$ is $k$-divisible, we derive a formula for the free cumulants of $x^k$ in terms of the free cumulants of $x$, involving $k$-divisible non-crossing partitions.
% Second, on the algebraic level, 

Second, we prove that if $a$ and $s$ are free and $s$ is $k$-divisible then $sps$ and $a$ are free, where $p$ is any polynomial (on $a$ and $s$) of degree $k-2$ on $s$. Moreover, we define a notion of R-diagonal $k$-tuples and prove similar results.
%and generalized for balanced $k$-tuples. %$(a_1,a_2,\dots a_n)$. 
%Formulas for free cumulants of $a=a_1a_2\cdots a_n$ and relations with $R$-cyclic matrices are also derived.

Next, we show that free multiplicative convolution between a measure concentrated in the positive real line and a probability measure with $k$-symmetry is well defined. Analytic tools to calculate this convolution are developed.
 
Finally, we concentrate on free additive powers of $k$-symmetric distributions and prove that $\mu^{\boxplus t}$ is a well defined probability measure, for all $t>1$. We derive central limit theorems and Poisson type ones. More generally, we consider freely infinitely divisible measures and prove that free infinite divisibility is maintained under the mapping $\mu\rightarrow\mu^k$.
% For these measures it is also shown that free additive convolution is well defined.%
 We conclude by focusing on ($k$-symmetric) free stable distributions, for which we prove a reproducing property generalizing the ones known for one sided and real symmetric free stable laws.

\end{abstract}
\newpage
\section*{Introduction}

Let $\mathcal{M}$ and $\mathcal{M}_\comp$ be the classes of all Borel probability measures on the real
line $\mathbb{R}$ and on the complex plane, respectively. Moreover, let $\mathcal{M}_{b}$ and $\mathcal{M}^{+}$ be the
subclasses of $\mathcal{M}$ consisting of probability measures with bounded support and of probability measures having support on $\mathbb{R}%
_{+}=[0,\infty)$, respectively.

For $q$ a primitive $k$-th root of unity consider the $k$-semiaxes $A_k:= \{ x\in \comp\mid x=tq^s $ for some $t>0$ and $s\in\nat\}$
and denote by $\mathcal{M}_k$ the subclass of $\mathcal{M}_\comp$ of probability measures supported $A_k$ such that $\mu(B)=\mu(qB)$, for all Borel sets $B$. A measure in $\mathcal{M}_k$ will be called $k$-symmetric.
We say that a measure in $\mathcal{M}_\comp$ has all moments if $m_{k}(\mu):=\int_{\mathbb{C}}\left\vert t\right\vert
^{n}\mu(\mathrm{d}t)<\infty,$ for each integer $n\geq1$.

In this paper we study random variables whose distribution is $k$-symmetric, which we will call \textit{$k$-divisible}.  We give a framework to these $k$-divisible random variables from the free probabilistic point of view. We consider various aspects of $k$-symmetric distributions including combinatorial, algebraic and probabilistic ones.

These $k$-divisible (non-commutative) random variables appear naturally in free probability. A typical example of a $k$-divisible random variable is the so called $k$-Haar unitaries with distribution $\mu=\frac{1}{k}\sum_{j=1}^k\delta_{q^j}$. $k$-divisible free random variables appear not only in the abstract setting but also in applications to random matrices. For instance, in \cite{Neagu} it is shown that an independent family $U_1,U_2,...,U_s$ of random $N\times N$ permutation matrices with cycle lengths of size $k$ converges in $*$-distribution to a $*$-free family $u_1,u_2,...,u_s$ of $k$-Haar unitaries. 

Other interesting examples of $k$-divisible free random variables come from the context of quantum groups. In Banica et al. \cite{BBCC}, where free Bessel laws are studied in detail, a modified $k$-symmetric version appears as the asymptotic law of the truncated characters of certain quantum groups. Similarly, from their studies of the law of the quantum isometry groups, Banica and Skalski \cite{BaS} found $k$-symmetric measures which are the analog of free compound Poissons, see Theorem 4.4 and Remark 4.5 in \cite{BaS}.

The free additive convolution $\boxplus$ and free multiplicative convolution $\boxtimes$ of measures supported on the real line (explained in Section 3) were introduced by Voiculescu \cite{Vo87b} to describe the sum and the product of free (non-commuting) random variables. These operations have many applications in the theory of large dimensional random matrices, since they allow to compute the asymptotic spectrum of the sum and the product of two independent random matrices from the individual asymptotic spectra \cite{HiPe00}, \cite{VoDyNi92}. Even though some work has been done in the physics literature (see e.g. \cite{BJN}) until now, this machinery could only be used for selfadjoint random variables and $k$-divisible random variables are not selfadjoint whenever $k>2$. Let us mention that $k$-symmetric distributions were considered by Goodman \cite{Goodman} in the framework of graded independence.

The Main Theorem (stated below) enables to define free multiplicative convolution between a measure concentrated in the positive real axis and a probability measure with $k$-symmetry. We extend the definition of the Voiculescu's $S$-transform to any $k$-symmetric measure $\mu$ to calculate effectively the free multiplicative convolution $\mu\boxtimes\nu$, between a $k$ symmetric measure $\mu$ and a measure $\nu$ supported on $\real^+$. 

 The Main Theorem also permits to define free additive powers for $k$-divisible measures leading to central limit theorems and Poisson type ones. Once we have free additive powers, the concept of free infinite divisibility arises naturally. We prove that for a $k$-symmetric measure $\mu$, free infinite divisibility is maintained under the mapping $\mu\rightarrow\mu^k$.

 Moreover, interesting combinatorial implications regarding the combinatorial convolution in $NC^k$, the poset of $k$-divisible non-crossing partitions are derived from the Main Theorem. This gives new ways of counting objects like $k$-equal partitions, $k$-divisible partitions and $k$-multichains both in $NC$ and $NC^k$.

From the combinatorial results on the poset of $k$-divisible non-crossing partitions we derive a formula for the free cumulants of $x^k$ in terms of the free cumulants of $x$ involving $k$-divisible non-crossing partitions.  Moreover, we define a notion of $R$-diagonal $k$-tuples and prove similar results.

%In\ Section 2 we review non-commutative random variables and their free products, and collect known results on the $S$-transform that are used later on.
%In section 3 We introduce the concept of $k$-divisible elements and study
%some of the combinatorial aspects of their cumulants....
%Section 4 is dedicated to the study of the
%$S$-transform of k-symmetric probability measures on $\mathbb{R}$ as well as to
%the free multiplicative convolution of symmetric distributions in
%$\mathcal{M}$ with distributions in $\mathcal{M}^{+}$.\ Finally, Section 4
%gives a description of the symmetric free stable distributions as the
%multiplicative convolution of the semicircle distribution with a positive free
%stable distribution.

A detailed description of the results of the paper is made in Section 1. Apart from this, the paper is organized as follows. The preliminaries needed in this paper are explained in Sections 2 and 3. In Section 2 we review non-commutative random variables and free probability including the analytic machinery to calculate free additive and multiplicative convolution, while in Section 3 we recall the combinatorics of non-crossing partitions.

We introduce the concept of $k$-divisible elements and study some of the combinatorial aspects of their cumulants in Section 4. Results of Section 4 are generalized in Section 5, where we introduce the concept of $R$-diagonal $k$-tuples. 
In Section 6, the main section, we present the main theorem of the paper and direct consequences, including free multiplicative convolution and free additive powers. Section 8 is dedicated to limit theorems: free central limit theorems, free compound Poisson, free infinite divisibility and connections to limit theorems in free multiplicative convolution is made. Finally, Section 9 deals with the case of unbounded measures, the $S$-transform of any $k$-symmetric probability measure as well as 
the free multiplicative convolution of distributions in
$\mathcal{M}_k$ with distributions in $\mathcal{M}^{+}$ is considered. We end by focusing on free stable distributions.

\section{Statement of Results}

First, from the combinatorial point of view we study the poset $NC^{k}(n)$ and its associated combinatorial convolution $*$ and translate the combinatorial convolution on $NC^k(n)$ to the convolution in $NC(n)$ of dilated sequences. Basically, we show that convolving $k$ times with the zeta-function in $NC$ is equivalent to convolving once with the zeta-function in $NC^k$.

\begin{thm}
The following statements are equivalent.
\begin{enumerate}[{\rm (1)}]
\item The multiplicative family $f:=(f_n)_{n>0}$ is the result of applying $k$ times the zeta-function to $g:=(g_n)_{n>0}$, that is 
\begin{equation*}
f=g\ast \underbrace{\zeta \ast \cdots \ast \zeta}_{k~times}
\end{equation*}

\item The multiplicative family $f^{(k)}:=(f_{n}^{(k)})_{n>0}$ is the result of applying one time the zeta-function to $g^{(k)}:=(g_{n}^{(k)})_{n>0}$, that is 
\begin{equation*}
f^{(k)}=g^{(k)}\ast \zeta,
\end{equation*}
where for a sequence $(a_n)_{n>0}$, the sequence $(a^{(k)}_n)_{n>0}$ denotes the dilated sequence given by $a^{(k)}_{kn}=a_n$ and $a^{(k)}_n=0$ if $n$ is not a multiple of $k$.
\end{enumerate}
\end{thm}

 Noticing that, when $x$ is $k$-divisible, the moments of $x$ are nothing else than the dilation of the moments of $x^k$ and using the so called moment-cumulant formula of Speicher (see e.g. \cite{NiSp06})  which relates the moments and the free cumulants via the combinatorial convolution in $NC(n)$ we give a relation between the free cumulants of $x$ and $x^k$ which generalizes results in \cite{NiSp95}. 

\begin{thm}
Let $(\mathcal{A},\phi )$ be a non-commutative probability
space and let $x$ be a $k$-divisible element with $k$-determining sequence $
(\alpha _{n}=\kappa_{kn}(x,...,x))_{n\geq 1}$. Then the following formula holds for the free cumulants
of $x^{k}$.
\begin{equation}
\kappa_{n}(x^{k},x^{k},...,x^{k})=[\alpha\ast \underbrace{\zeta \ast \cdots \ast \zeta}_{k~times}]_n.
\end{equation}
\end{thm}

Second, we consider how freeness is behaved when conjugating with $k$-divisible elements in a non-commutative probability space. More precisely, if $a$ and $s$ are free and $s$ is $k$-divisible then $a$ is aldo free from $sps$, where $p$ is any polynomial on $a$ and $s$ of degree $k-2$ on $s$. 

Moreover, we generalize the concept of diagonally balanced pairs from Nica and Speicher \cite{NiSp95}, which contains three of the most frequently used examples in free probability, that is, semicircular, circular and Haar unitaries, and prove similar results for what we call diagonally balanced $k$-tuples.

\begin{thm}
 Let $(\mathcal{A},\phi)$ be a non-commutative probability space, and let $(s_1,... ,s_k)$ be
a diagonally balanced $k$-tuple free from $a$. Moreover, let $h=s_1a_2s_2a_3s_3\cdots s_{k-1}a_{k-1}s_k$, where for all $i=1,...,n$ the element $a_i$ is free from $\{s_1,\cdots ,s_k\}.$ Then $h$ and $a$ are free.
\end{thm}

Furthermore, we realize $k$-divisible random variables as $R$-cyclic matrices \cite{NSS} with diagonally balanced $k$-tuples as entries. 

The third part of the paper deals with probability measures with $k$-symmetry and free convolutions $\boxtimes$ and $\boxplus$. Given a $k$-symmetric probability measure $\mu$ on $\mathcal{M}_k$, let $\mu^{k}$ be the probability measure in $\mathcal{M}^{+}$ induced by the map $t\rightarrow t^{k}$. In other words if $x$ is a $k$-divisible element with distribution $\mu$, then $\mu^{k}$ is the distribution of $x^k$.

One of the main results of this paper is to show that it is possible to define a free multiplicative $\mu\boxtimes \nu$ convolution between a probability $\mu$ in $\mathcal{M}^{+}$ and $k$-symmetric distribution $\nu$. 

The Main Theorem, which enables to define this free multiplicative convolution is the following.
\begin{mthm}\label{MainTheorem}
 Let $x,y\in(\mathcal{A},\phi)$ with $x$ positive and $y$ a $k$-divisible element. Consider $x_1,...,x_k$ positive elements with the same moments as $x$. Then $(xy)^k$ and $y^kx_1\cdots x_k$ have the same moments, i.e.
\begin{equation}
 \phi((xy)^{kn})=\phi((y^kx_1\cdots x_k)^n)
\end{equation}
\end{mthm}

As a byproduct we show that this free multiplicative convolution gives a $k$-symmetric distribution satisfying the relation $(\mu\boxtimes\nu)^k=\mu^{\boxtimes k}\boxtimes\nu^k$. Using this identity we give a formula for the moments of $\mu^{\boxtimes k}$ of positive measure in terms of $k$-divisible partitions.

An important analytic tool for computing the free multiplicative convolution of two probability measures is Voiculescu's $S$-transform. It was introduced in \cite{Vo87b} for non-zero mean distributions with bounded support and further studied by Bercovici and Voiculescu \cite{BeVo93} in the case of probability measures in $\mathcal{M}^{+}$ with unbounded support, see also \cite{BeVo92}. 

Raj Rao and Speicher \cite{RS} extended the $S$-transform to the case of random variables having zero mean and all moments. Their main tools are combinatorial arguments based on moment calculations.

We use the approach of \cite{RS} to extend the $S$-transform to random variables with first $k$ moments vanishing. After this, we specialize
to the case of $k$-divisible random variables where simple relations between the $S$-transforms
of $x$ and $x^k$ are found.

Moreover, for the case of $k$-symmetric probability measures we are able to extend the $S$-transform  even if we have no moments. To do this, we follow an analytic approach similar to  \cite{APA} and show that this $S$-transform allows to compute the desired free multiplicative convolution between probability measures on $[0,\infty)$ general
$k$-symmetric measures.

Another remarkable consequence of the Main Theorem is that we can define free additive powers $\mu^{\boxplus t}$ for $t>1$ when $\mu$ is a $k$-symmetric distribution. This opens the possibility to new central limit theorems.

 \begin{thm}[Free central limit theorem for $k$-symmetric measures]
Let $\mu$ be a $k$-symmetric measure with finite moments and $\kappa_k(\mu)=1$ then, as $N$ goes to infinity,
\begin{equation*}
D_{N^{-1/k}}(\mu^{\boxplus N})\rightarrow s_k,
\end{equation*}
where $s_k$ is the only $k$-symmetric measure with 
free cumulant sequence $\kappa_n(s_k)=0$ for all $n\neq k$ and $\kappa_k(s_k)=1$. Moreover,
 $$(s_k)^k=\pi^{\boxtimes k-1},$$
 where $\pi$ is a free Poisson measure with parameter 1.
\end{thm}

Free compound Poisson distributions exists in $\mathcal{M}_k$ and Poisson limit theorems also hold. We generalize Theorem 7.3 in \cite{BBCC}, where $\nu=\frac{1}{k}\sum_{j=1}^k\delta_{q^j}$ was considered in connection with free Bessel laws.
\begin{thm}
Let $\nu$ be a $k$-symmetric distribution, then the Poisson type limit convergence holds
\begin{equation*}
((1-\frac{\lambda}{N})\delta_0+\frac{\lambda}{N}\nu)^{\boxplus N}\rightarrow \pi(\lambda,\nu).
\end{equation*}
\end{thm}

We also address questions of free infinite divisibility. A measure is $\mu\in\mathcal{M}_k$ is said to be infinitely divisible if $\mu^{\boxplus t}\in\mathcal{M}_k$ for all $t>0$. For these measures, it is also shown that free additive convolution is well defined. Moreover we show that $\mu^k$ is also freely infinitely divisible.

\begin{thm} \label{freeinf10-2}
If $\mu$ is $k$-symmetric and $\boxplus$-infinitely divisible, then $\mu ^{k}$ is
also $\boxplus$-infinitely divisible.
\end{thm}

Finally, as an important example of distributions without finite moments and unbounded supports we consider free stable laws and show reproducing properties similar to the ones found in \cite{APA} and \cite{Be-Pa}.

\begin{thm}\label{Tstable-2}
For any $s,r>0$, let $\sigma^k_{1/(1+r)}$ be a $k$-symmetric strictly stable distribution of index $1/(1+r)$ and $\nu_{1/(1+s)}$ be a positive strictly stable distribution of index $1/(1+s)$. Then
\begin{equation}
\sigma^k_{1/(1+t)}\boxtimes\nu_{1/(1+s)}= \sigma^k_{1/(1+t+s)}.                           \end{equation}
\end{thm}

\section{Preliminaries on non-crossing partitions}

\subsection{Basic properties and definitions}

\begin{defi}

(1) We call $\pi =\{V_{1},...,V_{r}\}$ a \textbf{partition }of the set $[n]:=\{1, 2,.., n\}$
if and only if $V_{i}$ $(1\leq i\leq r)$ are pairwise disjoint, non-void
subsets of $S$, such that $V_{1}\cup V_{2}...\cup V_{r}=\{1, 2,.., n\}$. We call $%
V_{1},V_{2},..,V_{r}$ the \textbf{blocks} of $\pi $. The number of blocks of 
$\pi $ is denoted by $\left\vert \pi \right\vert $.

(2) A partition $\pi =\{V_{1},...,V_{r}\}$ is called \textbf{non-crossing} if for all $1 \leq a < b < c < d \leq n$
if $a,c\in V_{i}$ then there is no other subset $V_{j}$ with $j\neq i$ containing $b$ and $d$.

 (3) We say that a partition $\pi$  is  \textbf{$k$-divisible} if the size of all the blocks is multiple of  $k$.  If all the block are exactly of size k we say that $\pi$ is \textbf{$k$-equal}.

\end{defi}

We will denote the set of non-crossing partitions of $[n]$ by $NC(n)$, the set of $k$-divisible non-crossing partitions of $[kn]$ by $NC^k(n)$ and the set of $k$-equal non-crossing partitions of $[kn]$ by $NC_k(n)$\footnote{%
The notation that we follow is the one of Armstrong \cite{Arm09} which does not coincide with the one in Nica and Speicher \cite{NiSp06} for $2$-equal partitions.
\par
{}}.

\begin{rem} \label{rem1}The following characterization of non-crossing partitions is sometimes useful: for any $\pi\in NC(n)$, one can always find a block $V=\{r+1,\dots,r+s\}$ containing consecutive numbers. If one removes this block from $\pi$, the partition $\pi\setminus V\in NC(n-s)$ remains non-crossing.
\end{rem}

There is a graphical representation of a partition $\pi $ which makes clear the property of being crossing or non-crossing, usually called the circular representation. We think of $[n]$ as labelling the vertices of a regular $n$-gon, clockwise.
If we identify each block of $\pi $ with the convex hull of its
corresponding vertices, then we see that $\pi $ is non-crossing precisely
when its blocks are pairwise disjoint (that is, they don't cross).

\begin{figure}[here]
\begin{center}
\epsfig{file=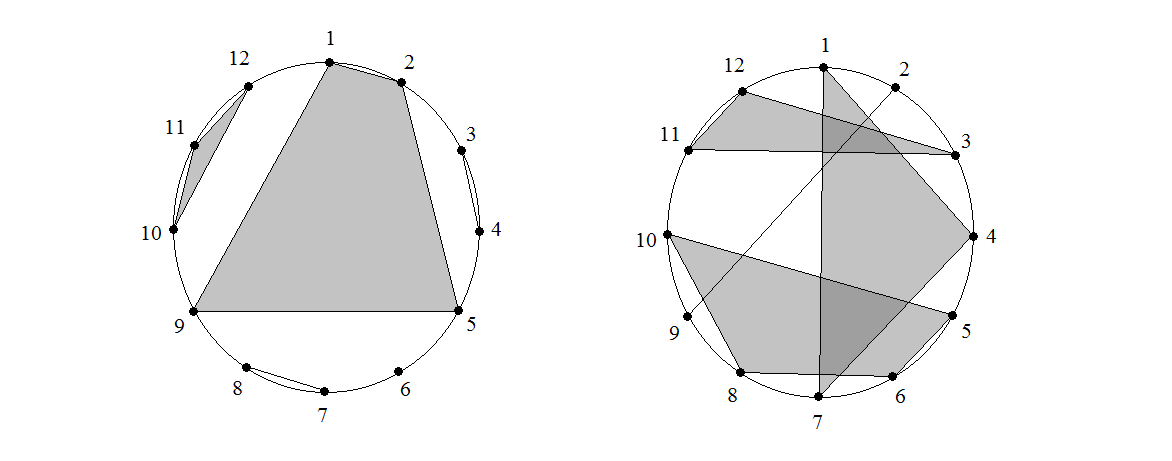, width=12cm}
\caption{Crossing and Non-Crossing Partitions}
\end{center}
\end{figure}

Figure 1 shows  the non-crossing partition $\{\{1,2,5,9\},\{3,4\},\{6\},\{7,8\},\{10,11,12\}\}$ of the set $[12]$, and the crossing
partition $\{ \{ 1,4,7\},\{ 2,9\}, \{ 3,11,12\} ,\{ 5,6,8,10\} \} $ of $[12]$  in their circular
representation.

It is well known that the number of non-crossing partition is given by the Catalan numbers $\frac{1}{n+1}\binom{2n}{n}$. More generally we can count $k$-divisible partitions, see \cite{Edel1}.

\begin{prop}
\label{k-divisible}Let $NC^{k}(n)$ be the set of non-crossing partitions
of $[nk]$ whose sizes of blocks are multiples of $k$. Then 
\begin{equation*}
\#NC^{k}(n)=\frac{\binom{(k+1)n}{n}}{kn+1}
\end{equation*}
\end{prop}

On the other hand, we can easily count $k$-equal partitions. 
\begin{cor}
\label{PartexactK}Let $NC_{k}(n)$ be the set of non-crossing partitions
of $nk$ whose blocks are of size of $k$. Then 
\begin{equation*}
\#NC_{k}(n)=\frac{\binom{kn}{n}}{(k-1)n+1}
\end{equation*}
\end{cor}

\begin{figure}[here]
\begin{center}
\epsfig{file=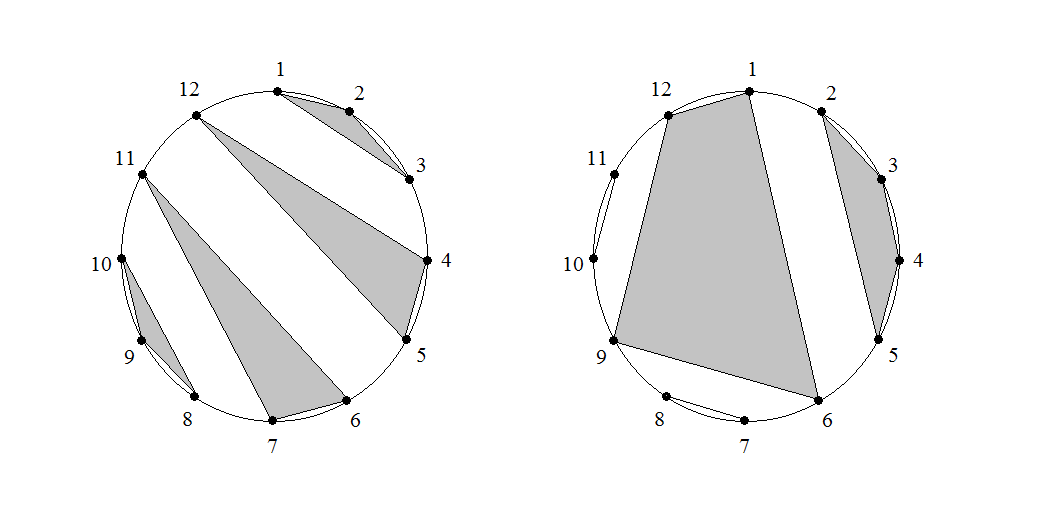, width=12cm}
\caption{3-equal and 2-divisible non-crossing partitions}
\end{center}
\end{figure}

The set $NC(n)$ can be equipped with the partial order $\preceq$ of reverse refinement ($\pi\preceq\sigma$ if and only if every block of $\pi$ is completely contained in a block of $\sigma$), making 
it a lattice.

\begin{defi} 
Given a partial order set, a $k$-multichain (or multichain of length $k-1$) is a sequence $x_{0}\leq
x_{1}\leq \cdots \leq x_{k-1}$ of elements of $P$.
We denote by $NC^{[k]}(n)$ the set of $k$-multichains in $NC(n)$.
\end{defi}

The number of $k$-multichains in $NC(n)$ was given by Edelman in \cite{Edel1}.

\begin{prop}
Let $NC^{[k]}(n)$ be the set of $k$-multichains in $NC(n)$. Then 
\begin{equation*}
\#NC^{[k]}(n)=\frac{\binom{(k+1)n}{n}}{kn+1}
\end{equation*}
\end{prop}

\begin{rem}[Definition of Kreweras complement]\label{kreweras}
Let $\pi $ be a partition in $NC(n).$ Then the Kreweras complement $K(\pi )$
is characterized in the following way. It is the only element $%
\sigma \in NC(\overline{1},2,...\overline{n})$ with the properties that $\pi
\cup \sigma \in NC(1,\overline{1},2,\overline{2},...,n,\overline{n}%
\}\backsimeq NC(2n)$ is non-crossing and that 
\begin{equation*}
\pi \cup \sigma \vee \{(1,\overline{1}),(2,\overline{2}%
),...,(n,\overline{n})\}=1_{2n}
\end{equation*}
The map $Kr:NC(n)\to NC(n)$ is an order reversing isomorphism. Furthermore, for all $\pi \in NC(n)$ we have that $|\pi|+|Kr(\pi)|=n+1$, see \cite{NiSp06} for details.
\end{rem}

The reader may have noticed from Proposition \ref{k-divisible} and Corollary \ref{PartexactK} that the number of $(k+1)$-equal non-crossing partitions of $n(k+1)$ and the number of $k$-divisible non-crossing partitions of $nk$ coincide with the number of $k$-multichains on $NC(n)$. This will be of relevance for this paper, and we will give a proof in Example \ref{k-k-k} as an application on how the zeta-function in $NC^{k}(n)$ is related to $\zeta ^{* k}$ in $NC(n).$

\subsection{Incidence algebra in $NC$}

Let us recall the main concepts about posets and incidence algebras first introduced by Rota et al. \cite{Rota}. The incidence algebra $I(P)=I(P,\mathbb{C})$ of a finite poset $(P,\leq)$ consists of all functions $f:P^{(2)}\rightarrow \mathbb{C}$ such that $f(\pi ;\sigma )=0$ whenever $\pi \nleq \sigma $ We can also
consider functions of one element; these are restrictions of functions of
two variables as above to the case where the first argument is equal to
$0$, i.e. $f(\pi )=f(0,\pi )$ for $\pi \in P$.

We endow $I(P,\mathbb{C})$ with the usual structure of vector space over $\mathbb{C}.$ On this incidence algebra we have a canonical multiplication or (combinatorial) convolution\footnote{Not to be confused with the concept of convolution of measures.
\par
{}} defined by

\begin{equation*}
(F\ast G)(\pi,\sigma ):=\sum\limits_{\substack{ \rho \in P  \\ \pi \leq \rho \leq \sigma 
}}F(\sigma, \rho )G(\rho ,\sigma )\text{.}
\end{equation*}

Moreover, for functions $f:P\rightarrow \mathbb{C}$ and $G:P^{(2)}\rightarrow 
\mathbb{C}$ we consider the convolution $f\ast G:P\rightarrow $ defined by%
\begin{equation*}
(f\ast G)(\sigma ):=\sum\limits_{\substack{ \rho \in P \\ \rho \leq \sigma 
}}f(\rho )G(\rho ,\sigma )\text{.}
\end{equation*}

The convolutions defined above are associative and distributive with respect to taking linear combinations of functions in $
P^{(2)}$ or in $P$. It is easy to verify that the function $\delta
:P^{(2)}\rightarrow 
\mathbb{C}$ defined as 
\[ \delta (\pi ,\sigma ) = \left\{ \begin{array}{cc}

       1      &  \pi=\sigma       \\
       0      &  \pi \neq \sigma   
    \end{array} \right. \]
is the unity with respect to the convolutions, making $I(P,\mathbb{C})$ a unital algebra.
Two other prominent functions in the in incidence algebra $I(P,\mathbb{C})$ are the zeta-function and its inverse the M\"{o}bius function. 

\begin{defi}
\label{Zeta-Mob}Let $(P,\leq)$ be a finite partially ordered set. The \textbf{zeta function} 
of $P$, $\zeta :P^{(2)}\rightarrow 
\mathbb{C}$ is defined by 
\begin{equation*}
\zeta (\pi ,\sigma )=1,\qquad \text{\ for all }\pi \leq \sigma \in P
\end{equation*}%
The inverse of $\zeta $ under the convolution is called the \textbf{M\"{o}bius function} of $P$, which will be denoted by $\mu $.
\end{defi}

\begin{rem}\label{multichains1}
Note that
\begin{equation*}
\zeta \ast \zeta (\pi ,\sigma )=\sum\limits_{\pi \leq \rho \leq \sigma
}1=card[\pi ,\sigma ],
\end{equation*}
and, more generally, 
\begin{equation*}
\underset{k\text{ times}}{\underbrace{\zeta \ast \zeta \ast \cdots \ast
\zeta }}(\pi ,\sigma )=\sum\limits_{\pi =\rho _{0}\leq \rho _{1}\cdots \leq
\rho _{k}=\sigma }1
\end{equation*}%
counts the number of $k+1-$multichains from $\pi $ to $\sigma.$ 
\end{rem}

\begin{defi}
Let $(\alpha _{n})_{n\geq 1}$ be a sequence of complex numbers. Define a family of functions $f_{n}:NC(n)\rightarrow 
\mathbb{C},$ $n\geq 1,$ by the following formula: if $\pi =\{V_{1},...,V_{r}\}\in
NC(n) $ then%
\begin{equation*}
f_{n}(\pi )=\alpha _{\left\vert V_{1}\right\vert }\cdots \alpha _{\left\vert
V_{r}\right\vert }\text{.}
\end{equation*}%
Then $(f_{n})$ is called the \textbf{multiplicative family of functions} on $%
NC$ determined by $(\alpha _{n})_{n\geq 1}$
\end{defi}

To emphasize the fact that the $\alpha _{n}$ encode the information of the
multiplicative family of function $f_{n}$ we will use the following notation.

\begin{notation}
Let $(\alpha _{n})_{n\geq 1}$ be a sequence of complex numbers, and let $%
(f_{n})$ is the multiplicative family of functions on $NC$ determined by $%
(\alpha _{n})_{n\geq 1}$. Then we will use the notation 
\begin{equation*}
\alpha _{\pi }:=f_{n}(\pi )\text{ \ \ \ for }\pi \in NC(n),
\end{equation*}%
and we will call the family of numbers $(\alpha _{\pi })_{n\in \mathbb{N}},_{\pi \in NC(n)}$ the multiplicative extension of $(\alpha _{n})_{n\in \mathbb{N}}$.
\end{notation}

Finally, for $g:=(g_n)_{n\geq1}$ and $f:=(f_n)_{n\geq1}$ multiplicative families in the lattice of non-crossing partitions we can define the combinatorial convolution $f*g:=((f*g)_n)_{n\geq1}$ in $NC$ by the following formula
\begin{equation*}
(f\ast g)_{n}:=\sum\limits_{\substack{ \pi \in NC(n) } 
} f_n(\pi )g_n(K(\pi)).
\end{equation*}

The importance of this combinatorial convolution is that the multiplicative family  $((f\ast g)_{n})_{n\geq1}$ can be used to describe free multiplicative convolution, in the following sense, see Equation (\ref{cum-prod}):
\begin{equation*}
\kappa_{n}(ab)=\sum_{\pi \in NC(n)}k_{\pi }(a)k_{K(\pi )}(b).
\end{equation*}
Moreover, the so-called moment-cumulant formula (see Equation (\ref{freemc})) may be stated as follows:
\begin{equation}\label{m-c formula}
m_{n}(x)=\sum_{\pi \in NC(n)}\kappa_{\pi }(a)
\end{equation}
which in our notation (if $m:=m_n(x)$ and $\kappa:\kappa_n(x)$) is nothing else than $m=k*\zeta$ or $k=m*\mu$.
There is a functional equation for the power series two multiplicative
families $(f_{n})_{n\geq 1}$ and $(g_{n})_{n\geq 1}$ on $NC,$ related by 
\begin{equation}
g=f\ast \mu\text{ }(\text{or equivalently: }f=g\ast
\zeta).  \label{mob5}
\end{equation}%
This is the content of next proposition.
\begin{prop}
\label{FunctionalEquation}Let $(f_{n})_{n\geq 1}$ and $(g_{n})_{n\geq 1}$ be
two multiplicative families on $NC$, which are related as in Equation (\ref{mob5}). Let $(\alpha _{n})_{n\geq 1}$ and $(\beta _{n})_{n\geq 1}$ be the
sequences of numbers that determine the multiplicative families; that is, we
denote $\alpha _{n}:=f_{n}(\mathbf{1}_{n})$ and $\beta _{n}:=g_{n}(\mathbf{1}%
_{n}),$ $n\geq 1.$ If we consider the power series%
\begin{equation*}
A(z)=1+\sum_{n=1}^{\infty }\alpha _{n}z^{n}\text{ \ \ \ and \ \ }%
B(z)=1+\sum_{n=1}^{\infty }\beta _{n}z^{n}\text{.}
\end{equation*}%
Then $A$ and $B$ satisfy the functional equation 
\begin{equation*}
A(z)=B(zA(z))\text{ \ \ \ and \ \ }B(z)=A(\frac{z}{B(z)})
\end{equation*}
\end{prop}

\section{Preliminaries on Free Probability}

Following \cite{VoDyNi92}, we recall that a non-commutative probability space
$(\mathcal{A},\varphi)$ is called a $W^{\ast}$-\textit{probability space} if
$\mathcal{A}$ is a non-commutative von Neumann algebra and $\varphi$ is a
normal faithful trace. A family of unital von Neumann subalgebras $\left\{
\mathcal{A}_{i}\right\}  _{i\in I}\mathcal{\subset A}$ in a $W^{\ast}%
$-probability space is said to be \textit{free} if $\varphi(a_{1}a_{2}%
\cdot\cdot\cdot a_{n})=0$ whenever $\varphi(a_{j})=0,a_{j}\in\mathcal{A}%
_{i_{j}},$ and $i_{1}\neq i_{2}\neq...\neq i_{n}.$ A self-adjoint operator
$X$\textit{\ is said to be affiliated with} $\mathcal{A}$ if $f(X)\in
\mathcal{A}$ for any bounded Borel function $f$ on $\mathbb{R}$. In this case
it is also said that $X$ is a (non-commutative) \textit{random variable}.
Given a self-adjoint operator $X$ affiliated with $\mathcal{A}$, the
\textit{distribution} of $X$ is the unique measure $\mu_{X}$ in $\mathcal{M}$
satisfying
\[
\varphi(f(X))=\int_{\mathbb{R}}f(x)\mu_{X}(\mathrm{d}x)
\]
for every Borel bounded function $f$ on $\mathbb{R}$.  If $\left\{
\mathcal{A}_{i}\right\}  _{i\in I}$ is a family of free unital von Neumann
subalgebras and $X_{i}$ is a random variable affiliated with $\mathcal{A}_{i}$
for each $i\in I$, then the \textit{random variables} $\left\{  X_{i}\right\}
_{i\in I}$ are said to be \textit{free}.

\subsection{Free Additive Convolution and Free Additive Powers}

The \textit{Cauchy transform} of a probability measure $\mu$ on $\mathbb{R}%
$\ is defined, for $z\in\mathbb{C}\backslash\mathbb{R}$, by%
\begin{equation}
G_{\mu}(z)=\int_{\mathbb{R}}\frac{1}{z-x}\mu\left(  \mathrm{d}x\right)  .
\label{CaTr}%
\end{equation}
It is well known that $G_{\mu}$ is an analytic function in $\mathbb{C}%
\backslash\mathbb{R}$, $G_{\mu}:\mathbb{C}_{+}\rightarrow\mathbb{C}_{-}$ and
that $G_{\mu}$ determines uniquely the measure $\mu.$ The \textit{reciprocal
of the Cauchy} \textit{transform} is the function $F_{\mu}\left(  z\right)
:\mathbb{C}_{+}\rightarrow\mathbb{C}_{+}$ defined by $F_{\mu}\left(  z\right)
=1/G_{\mu}(z).$ It was proved in \cite{BeVo93} that there are positive numbers
$\eta$ and $M$ such that $F_{\mu}$ has a right inverse $F_{\mu}^{-1}$ defined
on the region
\begin{equation}
\Gamma_{\eta,M}:=\left\{  z\in\mathbb{C};\left\vert \operatorname{Re}%
(z)\right\vert <\eta\operatorname{Im}(z),\quad\operatorname{Im}(z)>M\right\}
. \label{RegFinv}%
\end{equation}
The \textit{Voiculescu transform of }$\mu$ is defined by
\begin{equation}
\phi_{\mu}(z)=F_{\mu}^{-1}(z)-z \label{VoTr}%
\end{equation}
on any region of the form $\Gamma_{\eta,M}$, where $F_{\mu}^{-1}$ is defined,
see \cite{BePa99}, \cite{BeVo93}. The \textit{free cumulant transform} is a
variant of $\phi_{\mu}$ defined as
\begin{equation}
\mathcal{C}_{\mu}^{\boxplus}\mathcal{(}z)=z\phi_{\mu}(\frac{1}{z})=zF_{\mu
}^{-1}\left(  \frac{1}{z}\right)  -1, \label{freecumtrans}%
\end{equation}
for $z$ in a domain $D_{\mu}$ $\subset$ $\mathbb{C}_{-}$ such that $1/z\in$
$\Gamma_{\eta,M}$ where $F_{\mu}^{-1}$ is defined, see \cite{BNT06}.

The \textit{free additive convolution} of two probability measures $\mu _{1},\mu _{2}$ on $ \mathbb{R}$ is defined as the probability
measure $\mu _{1}\boxplus \mu _{2}$ on $\mathbb{R}
$ such that $\phi _{\mu _{1}\boxplus \mu _{2}}\mathcal{(}z)=\phi _{\mu _{1}}
\mathcal{(}z)+\phi _{\mu _{2}}\mathcal{(}z)$ or equivalently
\begin{equation}
\mathcal{C}_{\mu _{1}\boxplus \mu _{2}}^{\boxplus }\mathcal{(}z)=\mathcal{C} _{\mu
_{1}}^{\boxplus }\mathcal{(}z)+\mathcal{C}_{\mu _{2}}^{\boxplus }\mathcal{(}z)
\label{DefAdtCon}
\end{equation}
for $z\in D_{\mu _{1}}\cap D_{\mu _{2}}$. It turns out that $\mu _{1}\boxplus \mu _{2}$ is the
distribution of the sum $X_{1}+X_{2}$ of two free random variables $X_{1}$ and $X_{2}$ having
distributions $\mu _{1}$ and $\mu _{2}$ respectively.

 Moreover the \textit{free additive powers} of a measure $\mu$ are the measures $\mu^{\boxplus t}$ such that $\phi _{\mu^{\boxplus t}
 }\mathcal{(}z)=t\phi _{\mu}(z)$, or equivalently 
\begin{equation}\label{DefAdtPow}
\mathcal{C}_{\mu ^{\boxplus t}}(z)=t\mathcal{C}_\mu(z),
\end{equation}
whenever they exist. For any measure $\mu$ on $\real$ and $t>1$ the free additive power $\mu^{\boxplus t}$ exists. 

\subsection{Free Multiplicative Convolution}

The\textit{\ free multiplicative operation} $\boxtimes$ of probability measures with bounded support is defined as follows, see \cite{BeVo93}. Let $\mu_{1},\mu_{2}$
be probability measures on $\mathbb{R},$ with $\mu_{1}\in\mathcal{M}^{+}$ and
let $X_{1},X_{2}$ be free random variables such that $\mu_{X_{i}}=\mu_{i}$.

Since $\mu_{1}$ is supported on $\mathbb{R}_{+},$ $X_{1}$ is a positive
self-adjoint operator and $\mu_{X_{1}^{1/2}}$ is uniquely determined by
$\mu_{1}.$ Hence the distribution $\mu_{X_{1}^{1/2}X_{2}X_{1}^{1/2}}$ of the
self-adjoint operator $X_{1}^{1/2}X_{2}X_{1}^{1/2}$ is determined by $\mu_{1}$
and $\mu_{2}.$ This measure is called the \textit{free multiplicative convolution} of $\mu_{1}$ and $\mu_{2}$ and it is
denoted by $\mu_{1}\boxtimes$ $\mu_{2}$. This operation on $\mathcal{M}$ is
associative and commutative.

\begin{defi}
Let $x$ be a random variable in some non-commutative probability space $(\mathcal{A},\phi)$ with $\phi(x)\neq 0$. Then its $S$-transform is defined as follows. Let $\chi$ denote the inverse under composition of the series 
\begin{equation}
\psi(z):=\sum^\infty _{n=1}\phi(x^n)z^n,
\end{equation}
then 
\begin{equation}
S_x(z):=\chi(z)\frac{1+z}{z}.
\end{equation}
Moreover, if $\mu$ is distribution of $x$, the $S$-\textit{transform} of $\mu$ is defined as
\[
S_{\mu}(z)=\chi(z)\frac{1+z}{z}\text{.}%
\]
\end{defi}

The following result shows the role of the $S$-transform as an analytic tool
for computing free multiplicative convolutions with bounded support. It was shown in \cite{Vo87a} for measures for measures in $\mathcal{M}^{+}$ with bounded support and generalized in \cite{BeVo93} for measures in $\mathcal{M}^{+}$ with unbounded support. We will postpone the discussion of the unbounded case for Section \ref{Analytic Aspects}

\begin{prop}
\label{StransfPos} Let $\mu_{1}$ and $\mu_{2}$ be probability measures with compact support in
$\mathcal{M}^{+}$ with $\mu_{i}\neq\delta_{0}$, $i=1,2.$ Then $\mu
_{1}\boxtimes$ $\mu_{2}\neq\delta_{0}$ and
\[
S_{\mu_{1}\boxtimes\mu_{2}}(z)=S_{\mu_{1}}(z)S_{\mu_{2}}(z)
\]
in that component of the common domain which contains $(-\varepsilon,0)$ for
small $\varepsilon>0.$ Moreover, $(\mu_{1}\boxtimes$ $\mu_{2})(\{0\})=\max
\{\mu_{1}(\{0\}),\mu_{2}(\{0\})\}.$
\end{prop}

\begin{prop}
\label{WeakConCon} Let $\left\{  \mu_{n}\right\}  _{n=1}^{\infty}$ and
$\left\{  \nu_{n}\right\}  _{n=1}^{\infty}$ be sequences of probability
measures in $\mathcal{M}^{+}$ converging to probability measures $\mu$ and
$\nu\ $in $\mathcal{M}^{+}$, respectively, in the weak* topology and such that
$\mu\neq\delta_{0}\neq\nu$. Then, the sequences $\{\mu_{n}\boxtimes\nu
_{n}\}_{n=1}^{\infty}$ converges to $\mu\boxtimes\nu$ in the weak* topology.
\end{prop}

The next proposition is a particular case of a recent result proved in
\cite{RS} for probability measures $\mu_{1},\mu_{2}$ on $\mathbb{R}$ with
all moments, when $\mu_{1}$ has zero mean and $\mu_{2}\in\mathcal{M}^{+}$.

\begin{prop}
\label{MCRajSp} Let $\mu_{1}$ be a compactly supported 
measure on $\mathbb{R}$ with zero mean and let $\mu_{2}\in$ $\mathcal{M}^{+}$ have compact
support, with $\mu_{i}\neq\delta_{0}$, $i=1,2.$ Then, $\mu_{1}\boxtimes$
$\mu_{2}\neq\delta_{0}$ and
\[
S_{\mu_{1}\boxtimes\mu_{2}}(z)=S_{\mu_{1}}(z)S_{\mu_{2}}(z).
\]
\end{prop}

\begin{rem} In \cite{RS}, the definition of the $S$-transform is not unique. We will explain this in detail in Section \ref{mainsection}.
\end{rem}

From (\ref{freecumtrans}) and the fact that $\Psi_{\mu}(z)=\frac{1}{z}G_{\mu
}(\frac{1}{z})-1$, one obtains the following relation observed in
\cite{NiSp97} between the free cumulant transform and the $S$-transform%
\begin{equation}
z=\mathcal{C}_{\mu}^{\boxplus}\left(  zS_{\mu}(z\right)  ). \label{freecumS0}%
\end{equation}

This equation holds for measures in $\mathcal{M}^{+}$ or in $\mathcal{M}_{b}$
with zero mean. It was suggested in \cite{RS} that (\ref{freecumS0}) may
be used to define $S$-transforms of general probability measures on
$\mathbb{R}$.

As it is readily seen from Equation (\ref{freecumS0}) free additive powers may also be described by the $S$-transform in the following way
\begin{equation}\label{S-free}
S_{\mu^{\boxplus t}}(z)=\frac{1}{t}S_{\mu}(z/t),
\end{equation}
while the $S$ transform of a dilation is given by
\begin{equation}\label{S-Dil}
S_{D_t(\mu)}(z)=\frac{1}{t}S_{\mu}(z),
\end{equation}
from where we can deduce the following equality (see \cite{BN})
\begin{equation}\label{mult-additive}
(\mu\boxtimes\nu)^{\boxplus t}=D_t(\mu^{\boxplus t}\boxtimes\nu^{\boxplus t})~~~~t>1.
\end{equation}

\subsection{Free cumulants}

A measure $\mu $ has \textit{all moments} if $m_{k}(\mu )=\int_{\mathbb{R}}t^{k}\mu (\mathrm{d}t)<\infty ,$ for each
integer $k\geq 1$. Probability measures with compact support have all moments.

The \textbf{free cumulants} $(\kappa_{n})$ were introduced by Roland Speicher in \cite{Sp94}, in his
combinatorial approach to Voiculescu's free probability theory. We refer the reader to the book of Nica and Speicher \cite{NiSp06}
for an introduction to this combinatorial approach. Let $\mu \in \mathcal{M}_{b},$
then the cumulants are the coefficients $\kappa_{n}=\kappa_{n}(\mu )$ in the series expansion
\begin{equation*}
\mathcal{C}_{\mu }^{\boxplus }\mathcal{(}z)=\sum\nolimits_{n=1}^{\infty }\kappa_{n}(\mu )z^{n}.
\end{equation*}

The relation between the free cumulants and the moments is described using the
lattice of non-crossing partitions $NC(n)$, namely,
\begin{equation}\label{freemc}
m_{n}(\mu )=\sum\limits_{\mathbf{\pi }\in NC(n)}\kappa_{\pi }(\mu ), 
\end{equation}
where $\pi \rightarrow \kappa_{\pi }$ is the multiplicative extension of the free cumulants to
non-crossing partitions, that is
\begin{equation*}
\kappa_{\pi }:=\kappa_{\left\vert V_{1}\right\vert }\cdots \kappa_{\left\vert V_{r}\right\vert }\qquad
\text{for}\qquad \pi =\{V_{1},...,V_{r}\}\in NC(n) \text{.}
\end{equation*}
Since free cumulants are just the coefficients of the series expansion of  $\mathcal{C}_{\mu }^{\boxplus }\mathcal{(}z)$,
we can describe free additive convolution as follows.
\begin{equation*}
\kappa_{n}\mathcal{(}\mu _{1}\boxplus \mu _{2})=\kappa_{n}\mathcal{(}\mu _{1})+\kappa_{n} \mathcal{(}\mu_{2})
\end{equation*}
and 
\begin{equation*}
\kappa_{n}(\mu ^{\boxplus t})=t\kappa_{n}(\mu ). 
\end{equation*}

Moreover, free multiplicative convolution may be described in terms of cumulants as
follows.
\begin{equation}
\kappa_{\mu _{1}\boxtimes \mu _{2}}\mathcal{(}z)=\sum\limits_{\mathbf{\pi }\in NC(n)}\kappa_{\pi
}\mathcal{(}\mu _{1})\kappa_{K(\pi )}\mathcal{(}\mu _{2}) \label{cum-prod}
\end{equation}
where $K(\pi )$ is the Kreweras complement defined in the Remark \ref{kreweras}. 
In fact, this formula is valid for any two free random variables $a,b\in
A$, not necessarily selfadjoint. For two random variables which are free the free cumulants  of $ab$ are given by
\begin{equation*}
\kappa_{n}(ab)=\sum_{\pi \in NC(n)}\kappa_{\pi }(a)\kappa_{K(\pi )}(b).
\end{equation*}

We will often use the more general \textbf{formula for product as arguments}, first proved by Krawczyk and Speicher \cite{KrSp}. For a proof see Theorem 11.12 in Nica and Speicher \cite{NiSp06} 
 
\begin{thm}[Formula for products as arguments]\label{products}
Let $(\mathcal{A},\phi)$ be a non-commutative probability space and let $(\kappa_n)_{n\in\mathbb{N}}$ be the corresponding  free cumulants. Let $m,n\in\mathbb{N}$ and $1\leq i(1)<i(2)\cdots<i(m)=n$ be given and consider the partition $$\hat0_m={(1,...,i(1),...,(i(m-1)+1,...,i(m))}\in NC(n)$$ and the random variables $a_1,...,a_n\in \mathcal{A}$
then the following equation holds:
\begin{equation}\label{fprod}
\kappa_m(a_1\cdots a_{i(1)},...,a_{i(m-1)+1}\cdots a_{i(m)})=\sum\limits_{\substack{ \mathbf{\pi }\in NC(n)  \\ \pi \vee \hat0_m =1_{n} 
}}\kappa_{\mathbf{\pi }}(a_1,...,a_n)
\end{equation}

\end{thm}

\section{Combinatorics in $k$-divisible Non-Crossing Partitions}

In this section we study the poset $NC^{(k)}(n)$ of $k$-divisible
non-crossing partitions and the combinatorial convolution associated to this poset.

\subsection{The poset $NC^{k}(n)$}

Recall that a partition $\pi \in NC(nk)$ is called $k$-divisible if the size
of each block in $\pi $ is divisible by $k.$ As we have done for
non-crossing partitions, we can regard the set $NC^{(k)}(n)$ as a subposet
of $NC(nk)$.

\begin{defi}
We denote by $(NC^{k}(n),\leq )$ the induced subposet of $NC(kn)$
consisting of partitions in which each block has cardinality divisible by $k$.
\end{defi}

This poset was introduced by Edelman \cite{Edel1}, who calculated many of its enumerative invariants. Observe that coarsening of partitions preserves the property of $k$-divisibility, hence the set of $k$-divisible non-crossing partitions form a join-semilattice. However $NC^{k}(n)$ is not a lattice
for $k>1$ since, in general, some elements $\pi ,\sigma \in NC^{(k)}(n)$ do
not have a meet in $NC^{k}(n)$ (for instance, two different elements of
the type $\lambda =(k,k,...,k)).$

%\begin{rem}
%We will denote by $1^\kappa_n:=1_{kn}={{1,2,...,kn}}$ the largest partition in $NC{kn}$.
%\end{rem}

Since $NC^{(k)}(n)$ is a finite poset we can define the incidence algebra $
I(NC^{(k)}(n),\mathbb{C})$.
Recall that for a poset $P$ and functions $f:P\rightarrow \mathbb{C}$ and $G:P^{(2)}\rightarrow \mathbb{C}$ the convolution $f\ast G:P\rightarrow 
\mathbb{C}$ is defined as
\begin{equation*}
(f\ast G)(\sigma ):=\sum\limits_{\substack{ \rho \in P  \\ \rho \leq \sigma 
}}f(\rho )G(\rho ,\sigma )\text{.}
\end{equation*}
In particular, when $P=NC^{k}(n)$ and $G$ is the zeta function $\zeta $
(in $NC^{k}(n)$) we have that 
\begin{equation*}
f\ast \zeta (\sigma )=\sum\limits_{\substack{ \pi \in NC^{(k)}(n)  \\ \pi
\leq \sigma }}f(\pi ).
\end{equation*}

We will be interested the case when $f(\pi )$ is part of a multiplicative
family on $NC$ and on $NC^k$. So let us define a multiplicative family on $NC^k$ in analogy to the case of $NC$. 

\begin{defi}
Let $(\alpha _{n})_{n\geq 1}$ be a sequence of complex numbers. Define a
family of functions $f_{n}^{[k]}:NC^{k}(n)\rightarrow \mathbb{C},$ $n\geq 1,$ by the 
following formula: if $\pi =\{V_{1},...,V_{r}\}\in
NC^{(k)}(n)$ then
\begin{equation*}
f_{n}^{[k]}(\pi )=\alpha _{\left\vert V_{1}\right\vert/k}\cdots \alpha
_{\left\vert V_{r}\right\vert/k}\text{.}
\end{equation*}%
Then $(f_{n}^{[k]})_{n\geq 1}$ is called the \textbf{multiplicative family
of functions} on $NC^{(k)}(n)$ determined by $(\alpha _{n})_{n\geq 1}$.
\end{defi}

Observe, on one hand, that if $\pi =$ $\{V_{1},...,V_{r}\}$ is a $k$-divisible partition 
then the value 
\begin{equation*}
f_{nk}(\pi )=\alpha _{\left\vert V_{1}\right\vert }\cdots \alpha
_{\left\vert V_{r}\right\vert }
\end{equation*}%
only depends on the $\alpha _{i}$'s such that $k$ divides $i$ and then we can choose arbitrarily the values of $
\alpha _{i}$ for $i$ not divisible by $k.$ In particular, we can choose them
to be $0.$

On the other hand, if $(f_{n})_{n\geq 1}$ is the multiplicative family on $
NC(n)$ determined by a sequence $(\alpha _{n})_{n\geq 1}$ such that $\alpha
_{i}=0$ when $i$ is not divisible by $k$ then for $\pi \notin NC^{k}(n)$
we have that $\alpha _{\left\vert V_{1}\right\vert }\cdots \alpha
_{\left\vert V_{r}\right\vert }=0$ and thus, in $I(NC(kn),\mathbb{C}),$ we have 
\begin{equation*}
(f\ast \zeta)_{nk}(\sigma )=\sum\limits_{\substack{ \pi \in NC(kn)  \\ %
\pi \leq \sigma }}f(\pi )=\sum\limits_{\substack{ \pi \in NC^{k}(n) 
\\ \pi \leq \sigma }}f_{nk}(\pi ).
\end{equation*}%
and
\begin{equation*}
(f^{[k]}\ast \zeta) _{nk}(\sigma )=\sum\limits_{\substack{ \pi \in NC^{k}(n) 
\\ \pi \leq \sigma }}f^{[k]}_{n}(\pi )=\sum\limits_{\substack{ \pi \in NC^{k}(n) 
\\ \pi \leq \sigma }}f_{nk}(\pi ).
\end{equation*}
So, for multiplicative families on $NC$ determined by sequences such that $
\alpha _{i}=0$ whenever $i$ is not divisible by $k,$ the convolution with
the zeta function $\zeta $ is exactly the same in $I(NC(kn),\mathbb{C}
)$ as in $I(NC^{(k)}(n),\mathbb{C})$. 

 Let us fix some notation to encode the information in sequences of this type.

\begin{notation}
We call a sequence $\alpha _{n}^{(k)}$ the $k$-dilation of $\alpha _{n}$ if $
\alpha _{kn}^{(k)}=\alpha _{n}$ and $\alpha _{n}^{(k)}=0$ if $n$ is not
multiple of $k$.
\end{notation}
By the arguments given above we can deal with the convolution between 
$(f^{[k]}_{n})_{n\geq 1}$ (a multiplicative family on $NC^{(k)}(n))$ and $\zeta\in I(NC^{(k)}(n),\mathbb{C})$ by just considering the $k$-dilations of the original sequence and work with the
usual convolution in $NC(n)$. In particular, we can use the functional
equation in Proposition \ref{FunctionalEquation} to get a functional
equation for multiplicative families in $NC^{(k)}(n)$.

\begin{prop}
\label{FunctionalEquation2}Let $g_{n}^{[k]}$ be a multiplicative family in $
NC^{(k)}(n)$ determined by the sequence $(\beta _{n})_{n\geq 1}$ and $
f_{n}^{[k]}$ be a multiplicative family in $NC^{(k)}(n)$ determined by the
sequence $(\alpha _{n})_{n\geq 1}$. Suppose that $f^{[k]}=g^{[k]}\ast
\zeta $. If we consider the power series%
\begin{equation*}
A(z)=1+\sum_{n=1}^{\infty }\alpha _{n}z^{n}\text{ \ \ and \ }%
B(z)=1+\sum_{n=1}^{\infty }\beta _{n}z^{n}\text{.}
\end{equation*}%
then 
\begin{equation*}
A(z)=B(zA(z)^{k})
\end{equation*}
\end{prop}

\begin{proof}
Since $f^{[k]}=g^{[k]}\ast
\zeta$ is equivalent to $f^{(k)}=g^{(k)}\ast \zeta \ $ then, by Proposition \ref{FunctionalEquation}, 
the power series $A_{k}(z)=1+\sum_{n=1}^{\infty
}\alpha _{n}^{(k)}z^{n}$ \ and \ $B_{k}(z)=1+\sum_{n=1}^{\infty }\beta
_{n}^{(k)}z^{n}$ are related by the functional equation 
\begin{equation*}
A_{k}(z)=B_{k}(zA_{k}(z)).
\end{equation*}%
Note that $A_{k}(z)=A(z^{k})$ and $B_{k}(z)=B(z^{k})$, hence 
\begin{equation*}
A(z^{k})=B(z^{k}A(z^{k})^{k}).
\end{equation*}%
Making the change of variable $z^{k}=y$ we get 
\begin{equation*}
A(y)=B(yA(y)^{k}).
\end{equation*}%
as desired$.$
\end{proof}

\subsection{Motivating example}

Consider the three following objects.
 \begin{enumerate}[{\rm (i)}]
\item $NC_{k+1}(n)$:
Non crossing partitions in $NC((k+1)n)$ with each block of size $k+1$.
\item $NC^{k}(n)$ :
Non crossing partitions of in $NC(kn)$ with block of size multiple of $k$. 
\item $NC^{[k]}(n)$ :
Multichains of order $k+1$ in $NC(n)$.
\end{enumerate}

It is well known that the Fuss-Catalan numbers count the three of them. Different ways to count them are now known. The first ones were counted by Kreweras \cite{Kr72}. Also bijections between them have been given in  \cite{Ar1} and \cite{Edel1} . Moreover in \cite{Arm09} an order has been given to (ii) makings the objects in (ii) and (iii) isomorphic as posets and generalized to other Coxeter groups.

We want to show we can use Proposition \ref{FunctionalEquation2} to derive the same functional equation for the three of them without counting them explicitly.

\begin{exa} \label{k-k-k}
\label{FunctionalEquation3}Denote the cardinality of $NC_{k}(n)$ by
the number $Z_{n}^{k}$. Let $(\beta _{n}=0)_{n\geq 2}$, $\beta _{1}=1$ and $(\alpha _{n})_{n\geq1}$ be two sequences with respective multiplicative
families $(g^{[k]}_n)_{n>0}$ and $(f^{[k]}_{n})_{n>0}$ on $NC^{k}$ related by the formula $f^{[k]}=g^{[k]}\ast \zeta$. Then  $\alpha_{n}$ equals $Z_{n}^{k}$ and $A(z)=A(z)=1+\sum_{n=1}^{\infty }\alpha _{n}z^{n}$ satisfies $$A(z)=1+zA(z)^{k}.$$ Indeed,
\begin{equation*}
\alpha _{n}=f_{n}^{(k)}(1_{nk})=g_{n}^{(k)}\ast \zeta (1_{nk})=\sum\limits 
_{\substack{ \pi \in NC_(nk)  \\ \pi \leq 1_{nk}}}g_{n}^{(k)}(\pi
)=\sum\limits_{\substack{ \pi \in NC_{k}(n)  \\ \pi \leq 1_{nk}}}%
1=Z_{n}^{(k)}.
\end{equation*}%
Then, by Proposition \ref{FunctionalEquation2} the power series $%
A(z)=1+\sum_{n=1}^{\infty }\alpha _{n}z^{n}$ and \ $B(z)=1+\sum_{n=1}^{%
\infty }\beta _{n}z^{n}$ are related by%
\begin{equation*}
A(z)=B(zA(z)^{k}).
\end{equation*}%
The power series for the sequence $(\beta _{n})_{n\geq
1} $ is $B(z)=1+z$
and then  
\begin{equation*}
A(z)=1+zA(z)^{k}.
\end{equation*}

\end{exa}

\begin{exa}
\label{FunctionalEquation4}Denote the cardinality of $NC^{k}(n)$ by
the number $C_{n}^{(k)}$. Let $(\beta _{n}=1)_{n\geq 1}$ and $(\alpha _{n})_{n>1}$ be two sequences with respective multiplicative families  on $NC^{k}$ related by the formula $f^{[k]}=g^{[k]}\ast \zeta$. Then  $f_{n}^{(k)}$ equals $C_{n}^{(k)}$ and 
\begin{equation*}
A(z)=1+zA(z)^{k+1}.
\end{equation*}
Indeed,
\begin{equation*}
\alpha _{n}=f_{n}^{(k)}(1_{nk})=g_{n}^{(k)}\ast \zeta (1_{nk})=\sum\limits 
_{\substack{ \pi \in NC^{k}(n)  \\ \pi \leq 1_{nk}}}g_{n}^{k}(\pi
)=\sum\limits_{\substack{ \pi \in NC^{(k)}(n)  \\ \pi \leq 1_{nk}}}%
1=C_{n}^{(k)}.
\end{equation*}%
Again, by Proposition \ref{FunctionalEquation2} the power series $%
A(z)=1+\sum_{n=1}^{\infty }\alpha _{n}z^{n}$ and \ $B(z)=1+\sum_{n=1}^{%
\infty }\beta _{n}z^{n}$ are related by%
\begin{equation*}
A(z)=B(zA(z)^{k}).
\end{equation*}
The power series for the sequence $(\beta _{n}=1)_{n\geq
1} $ is 
\begin{equation*}
B(z)=\sum_{n=0}^{\infty }z^{n}=\frac{1}{1-z}
\end{equation*}%
and then 
\begin{equation*}
A(z)=\frac{1}{1-zA(z)^{k}}
\end{equation*}%
or equivalently 
\begin{equation*}
A(z)=1+zA(z)^{k+1}.
\end{equation*}

\end{exa}

Finally, for $k$-multichains we have the following.

\begin{exa}
Let $c_{n}^{k}$ $(n,k\geq 1)$ denote the number of $k$-multichains in $%
NC(n).$ For every $k\geq 1$ let $(f_{n,k})_{n\geq 1}$ be the
multiplicative family of functions on $NC$ determined by the sequence $%
c_{n}^{k}$. As we have noticed in Remark \ref{multichains1}, for every poset $P$, the number of $(k+2)$-multichains from $\pi \in P$ to $\sigma \in P$ is given by%
\begin{equation*}
\underset{k+1\text{ times}}{\underbrace{(\zeta \ast \zeta \ast \cdots \ast
\zeta )}}(\pi ,\sigma )\text{. \ \ for all }k\geq 1
\end{equation*}%
In particular, for $NC(n),$ if we plug $\pi =\mathbf{0}_{n}$ and $\pi =%
\mathbf{1}_{n}$ we get the that the number of $(k+1)$-multichains is given by%
\begin{equation*}
\underset{k+1\text{ times}}{\underbrace{(\zeta _{n}\ast \zeta _{n}\ast
\cdots \ast \zeta _{n})}}(\mathbf{0}_{n},\mathbf{1}_{n})\text{ \ \ for all }%
n,k\geq 1\text{.}
\end{equation*}
In other words 
\begin{equation*}
f_{n,k}=\underset{k+1\text{ times}}{\underbrace{\zeta _{n}\ast \zeta
_{n}\ast \cdots \ast \zeta _{n}}}\text{ \ \ for all }n,k\geq 1,
\end{equation*}
or equivalently 
\begin{equation*}
f_{n,k+1}=f_{n,k}\ast \zeta _{n}\text{ \ \ for all }n,k\geq 1.
\end{equation*}
Now, consider for each $k\geq 1,$ the power series 
\begin{equation*}
A_{k}(z):=1+\sum_{n=1}^{\infty }c_{n}^{k}z^{n}\text{.}
\end{equation*}
From the Proposition \ref{FunctionalEquation},\ the power series $A_{k}(z)$
and $A_{k+1}(z)$ must satisfy the functional equation 
\begin{equation*}
A_{k+1}(z)=A_{k}(zA_{k+1}(z)).
\end{equation*}%
It is easy to see that power series of $c_{n}^{2}$ (the Catalan numbers) satisfy the relation%
\begin{equation*}
A_{1}(z)=1+zA_{1}(z)^{2}.
\end{equation*}%
By induction we see that $A_{k}$ satisfies the functional equation%
\begin{equation*}
A_{k}(z)=1+zA_{k}(z)^{k+1}.
\end{equation*}%
Indeed, since $A_{k}(y)=1+yA_{k}(y)^{k+1}$ plugging $y=zA_{k+1}$ we get 
\begin{eqnarray*}
A_{k+1}(z) &=&A_{k}(zA_{k+1}(z))=1+zA_{k+1}(z)A_{k}(zA_{k+1}))^{k+1}\\
&=&1+zA_{k+1}(z)(A_{k+1}(z))^{k+1} \\
&=&1+zA_{k+1}(z)^{k+2}.
\end{eqnarray*}
\end{exa}

We have seen that all of the three objects satisfy the same functional equation and then must be counted by the same sequence.
So the multichains of length $k+1$ in $
NC(n)$ are in bijection with the $k$-divisible non-crossing partitions in $NC(nk)$ 
and  with the $k+1$-equal partitions in $NC(n(k+1)$. This result is known and was already in \cite{Edel1} but we emphasize that our derivation did not use at any moment the explicit calculation of these object but rather relies on deriving a functional equation, these ideas will used later in this paper.
\begin{rem}
This bijection goes further. In fact, one can give an explicit order to $k$-multichains so that $NC^{k}(n)\backsimeq NC^{(k)}(n)$ as ordered sets. We will not give details about this but rather refer the reader to Chapters 3
and 4 of \cite{Arm09}. The  point here is that we may think of both as the same objects.
\end{rem}

%\begin{prop}
%The number of $k$-multichains in $NC(n)$ are counted by the Fuss-Catalan number 
%\begin{equation*}
%C_{n}^{(k)}=\frac{1}{kn+1}\binom{(k+1)n}{n}\text{.}
%\end{equation*}
%\end{prop}

\subsection{The convolution of $k$-dilated sequences in $NC$ }

Since the convolution with $\zeta $ in $NC^{(k)}(n)$
is equivalent to convolution with $\zeta $ in $NC(n)$ for sequences dilated
by $k$ we can forget about the former and focus on how
convolution with $k$-dilated sequences behave in $NC(n)$. From now on, we will
prefer to use the notation $\alpha _{\pi }=\alpha _{\left\vert
V_{1}\right\vert }\cdots \alpha _{\left\vert V_{r}\right\vert }$ instead of $
f(\pi )$ since there is no confusion.

%Recall a that a sequence $(\alpha _{n}^{(k)})_{n\geq 1}$ is the $k$-dilation of $(\alpha _{n})_{n\geq 1}$ if $\alpha _{kn}^{(k)}=\alpha _{n}$ and $\alpha_{n}^{(k)}=0$ if $n$ is not multiple of $k.$

The first result gives a relation between the formal power series of the $k$-dilation of the sequence $(m_{n})_{n\geq 1}$ and the $(k+1)$-dilation of
the sequence $(\beta _{n})_{n\geq 1},$ when the two sequences related $m_{n}=\beta
_{n}\ast \zeta $, namely,
\begin{equation*}
m_{n}=\sum\limits_{\pi \in NC(n)}\beta _{\pi }
\end{equation*}

\begin{prop}
\label{Main Proposition}Let $k$ be a positive integer and let 
\begin{eqnarray*}
A(z) &=&1+\sum \alpha _{n}z^{n} \\
B(z) &=&1+\sum \beta _{n}z^{n} \\
M(z) &=&1+\sum m_{n}z^{n}
\end{eqnarray*}
Then any two of the following three statements imply the third
\begin{enumerate}[{\rm (i)}]
\item 
 $M(z)=B(zM(z)).$
\item 
 $M(z)=A(zM(z)^{k}).$
\item 
 $B(z)=A(zB(z)^{k-1}).$
\end{enumerate}

Equivalently, each two of the following three statement imply the third.
\begin{enumerate}[{\rm (i)}]
\item 
The sequences $m_{n}$ and $\beta _{n}$ are related by the formula .
\begin{equation*}
m_{n}=\sum\limits_{\pi \in NC(n)}\beta _{\pi }
\end{equation*}
\item The sequences $\alpha _{n}$ and $m _{n}$ are related by \ the formula 
\begin{equation*}
m^{(k)}_{n}=\sum\limits_{\pi \in NC(kn)}\alpha _{\pi }^{(k)}
\end{equation*}%
\item The sequences $\alpha _{n}$ and $\beta _{n}$ are related by \ the
formula 
\begin{equation*}
\beta^{(k-1)} _{n}=\sum\limits_{\pi \in NC((k-1)n)}\alpha _{\pi }^{(k-1)}
\end{equation*}
\end{enumerate}
\end{prop}

\begin{proof}
(i) \& (ii) $\Rightarrow $ (iii). Evaluating in $B$ in $zM(z)$ we get 
\begin{equation*}
B(zM(z))=M(z)=A(zM(z)^{k})=A(zM(z)M(z)^{k-1})=A(zM(z)B(zM(z))^{k-1}),
\end{equation*}%
making the change of variable $y=zM(z)$ the result holds.

(i) \& (iii) $\Rightarrow $ (ii). The relation (i) is equivalent to $B(z)=M(z/B(z))$ so 
\begin{equation*}
M(z/B(z))=B(z)=A(zB(z)^{k})=A(\frac{z}{B(z)}B(z)^{s+1})=A(\frac{z}{B(z)}%
M(z/B(z))^{k}),
\end{equation*}%
making the change of variable $y=z/B(z)$ we get the result.

The last equality follows along the same lines. The equivalence of the next three statements in terms of sums on non-crossing partitions follows from
Proposition \ref{FunctionalEquation2}..
\end{proof}

We can use the last result recursively to get a formula for the $k$-fold convolution of with the zeta function $\zeta$.

\begin{cor}\label{maincor}
Let $M(z),A(z),B_{i}(z)$ formal power series and such that
\begin{enumerate}[{\rm (i)}]
\item 
 $M(z)=A(zM(z)^{k})$
\item 
$M(z)=B_{1}(z(M(z))$
\item 
 $B_{i}(z)=B_{i+1}(zB_{i}(z)),$ for $i=1,...,k-1.$
\end{enumerate}

Then $B_{i}(z)=A(zB_{i}(z)^{k-i})$, in particular $B_{n}(z)=A(z).$
\end{cor}

\begin{proof}
We will use induction on $i.$

For $i=1$, we use i) and ii) and Proposition \ref{Main Proposition} to get 
\begin{equation*}
B_{1}(z)=A(zB_{1}(z)^{k-1})\text{.}
\end{equation*}%
Now suppose that the statement is true for $i=n.$ Then%
\begin{equation*}
B_{n}(z)=A(zB_{n}(z)^{k-n})
\end{equation*}%
also by iii) $B_{n}(z)=B_{n+1}(zB_{n}(z)),$ so again by Proposition \ref{Main Proposition} we get%
\begin{equation*}
B_{n+1}(z)=A(zB_{n+1}(z)^{k-n-1}).
\end{equation*}
\end{proof}

The last proposition may look rather artificial. But it explains how the successive convolution with the zeta-function in $NC(n)$
is equivalent to the convolution with the zeta-function in $NC^{(k)}(n)$, as we state more precisely in the following theorem.

\begin{thm}
\label{MainRemark} The following statements are equivalent.
\begin{enumerate}[{\rm (1)}]
\item The multiplicative family $f:=(f_n)_{n>0}$ is the result of applying $k$ times the zeta-function to $g:=(g_n)_{n>0}$, that is 
\begin{equation*}
f=g\ast \underbrace{\zeta \ast \cdots \ast \zeta}_{k~times}
\end{equation*}

\item The multiplicative family $f^{(k)}:=(f_{n}^{(k)})_{n>0}$ is the result of applying one time the zeta-function to $g^{(k)}:=(g_{n}^{(k)})_{n>0}$, that is 
\begin{equation*}
f^{(k)}=g^{(k)}\ast \zeta.
\end{equation*}
\end{enumerate}
\end{thm}

\begin{proof}
This is just a reformulation of Corollary \ref{maincor} in terms of combinatorial convolution.
\end{proof}

\begin{rem} This phenomena is very specific for the non-crossing partitions. For instance, it does not occur if we change $NC(n)$ by $P(n)$ the lattice of all partition nor $IN(n)$ the lattice of interval partitions. 
\end{rem}
To finish this section let us see how Theorem \ref{MainRemark} may be applied in our motivating example to give a shorter proof of the fact that $k+1$-multichains,
$k$-divisible non-crossing partitions and $k+1$-equal non-crossing partitions have the same cardinality.
\begin{exa}
 Let ${a_n}$ be the sequence determined by $a_{1}=1$ and $a_{n}=0$ for $n>1$ (notice that this is just the sequence associated of the delta-function $\delta$). Next, consider $$c=a\ast \underbrace{\zeta \ast \cdots \ast \zeta}_{k+1~times}.$$
On one hand, by Remark \ref{multichains1}, $c_n$ counts the number of $(k+1)$-multichains of $NC(n)$.
On the other hand, by Theorem \ref{MainRemark} applied to $a_n$
\begin{equation*}
c_{n}=c^{(k+1)}_{(k+1)n}=\sum_{\pi \in NC((k+1)n)}a^{(k+1)}_{\pi }=\sum_{\pi \in NC_{k+1}(n)}1=\#NC_{k+1}(n)
\end{equation*}
and we get the number of $(k+1)$-equal noncrossing partitions.
Finally, for $k$-divisible partitions, consider $b=a*\zeta$. Then $$c=b\ast \underbrace{\zeta \ast \cdots \ast \zeta}_{k~times}$$
and 
$$b_{n}=\sum_{\pi \in NC(n)}a_{\pi }=1. $$
So, again by Theorem \ref{MainRemark}, applied to $b_n$,
$$c_n=c^{(k)}_{kn} =\sum_{\pi \in NC(nk)}b^{(k)}_{\pi }=\sum_{\pi \in NC^k(n)}1=\# NC^k(n).$$
Thus we have proved that $c_n$ counts $k$-divisible non crossing partitions of $[kn]$, $(k+1)$-equal non crossing partitions of $[(k+1)n]$ and $(k+1)$-multichains on $NC(n)$.
\end{exa}

We can push more this example to also recover Theorem $3.6.9$ of Armstrong \cite{Arm09}  for the case of classical $k$-divisible non-crossing partitions. The proof is left to the reader.
\begin{cor}
The number of $l$-multichains of $k$-divisible noncrossing partitions equals the number of $lk$ multichains of $NC(n)$ and is given by the Fuss-Catalan number $C_{kl,n}$. 
\end{cor}

It would be very interesting to see if similar arguments can be used to count invariants for non-crossing partitions in the different Coxeter groups. To the knowledge of the author this is not known.

\section{$k$-divisible elements}

We introduce the concept of $k$-divisible elements and study
some of the combinatorial aspects of their cumulants. The main result in this section describes the cumulants of the $k$-th power
of a $k$-divisible element.

\subsection{Basic properties and definitions}

Let $(\mathcal{A},\phi )$ be a non-commutative probability space.
\begin{notation}

1) An element $x\in \mathcal{A}$ is called $k$-divisible if the only non vanishing
moments are multiples of $k$. That is 
\begin{equation*}
\phi (x^{n})=0\text{ if }k\nmid n
\end{equation*}

2) Let $x\in \mathcal{A}$ be $k$-divisible and let $\alpha _{n}:=\kappa_{kn}(x,...,x)$. We
call $(\alpha _{n})_{n\geq 1}$ the $k$-determining sequence of $x$
\end{notation}

It is clear that $x\in \mathcal{A}$ is $k$-divisible if and only if its non-vanishing free
cumulants are multiples of $k$. The following is a generalization of Theorem $11.25$ in Nica and Speicher \cite{NiSp06} where, for an even element $x$, the free cumulants of $x^2$ are given in terms of the moments of $x$.

\begin{thm}[Free cumulants of $x^{n}$, First formula] \label{1formula}
\label{kumulants of x^k } Let $(\mathcal{A},\phi )$ be a non-commutative probability
space and let $x$ a $k$-divisible element with $k$-determining sequence $(\alpha _{n})_{n\geq 1}$. Then the following formula holds for the cumulants
of $x^{k}$.
\begin{equation}
\kappa_{n}(x^{k},x^{k},...,x^{k})=\sum\limits_{\mathbf{\pi }\in
NC((k-1)n)}\alpha _{\pi }^{(k-1)}.  \label{cumulant-cumulant formula}
\end{equation}
\end{thm}

\begin{proof}[First proof]
Set $\alpha _{n}=\kappa_{kn}(x)$, $\beta _{n}=\kappa_{n}(x^{k},...,x^{k}),$ $%
m_{n}=m_{n}(x^{k})=m_{kn}(x)$ and let 
\begin{eqnarray*}
A(z) &=&1+\sum \alpha _{n}z^{n} \\
B(z) &=&1+\sum \beta _{n}z^{n} \\
M(z) &=&1+\sum m_{n}z^{n}
\end{eqnarray*}%
The moment-cumulant formula for $x^{k}$ gives 
\begin{equation*}
M(z)=B_{1}(z(M(z))
\end{equation*}%
and the moment-cumulant formula for $x$ says
\begin{equation*}
M(z)=A(zM(z)^{k})
\end{equation*}%
so by Proposition \ref{Main Proposition} we get 
\begin{equation*}
B(z)=A(zB(z)^{k-1})
\end{equation*}%
or equivalently. 
\begin{equation*}
\kappa_{n}(x^{k},x^{k},...,x^{k})=\sum\limits_{\mathbf{\pi }\in
NC((k-1)n)}\alpha _{\pi }^{(k-1)}
\end{equation*}
\end{proof}
\begin{proof}[Second proof]
This proof is more involved but gives a better insight on the combinatorics of $k$-divisible elements and works for the more general setting of diagonally balances $k$-tuples of Section \ref{k tuples}.
The argument is very similar as in the proof in \cite{NiSp06} for $k=2$.
The formula for products as arguments (Eq. \ref{fprod}) yields
\begin{equation*}
\kappa_{n}(x^{k},x^{k},...,x^{k})=\sum\limits_{\substack{ \mathbf{\pi }\in
NC(kn)  \\ \pi \vee \sigma =1_{kn}}}\kappa_{\mathbf{\pi }}(x,x,...,x,x)
\end{equation*}%
with $\sigma =\{(1,2,3,..k),(k+1,k+2,...,2k),...,(kn-n+1,...,kn)\}$.

Observe that since $x$ is $k$-divisible then 
\begin{equation*}
\sum\limits_{\substack{ \mathbf{\pi }\in NC(kn)  \\ \pi \vee \sigma =1_{kn} 
}}\kappa_{\mathbf{\pi }}(x,x,...,x,x)=\sum\limits_{\substack{ \mathbf{\pi }\in
NC(kn),\text{ }\pi \text{ k-divisible}  \\ \pi \vee \sigma =1_{kn}}}\kappa_{\pi
}[x,x,...,x,x]
\end{equation*}
The basic observation is the following
\begin{eqnarray*}
\{\pi &\in &NC(kn)\mid \pi~k \text{-divisible},\pi \vee \sigma =1_{kn}\}= \\
\{\pi &\in &NC(kn)\mid \pi ~k\text{-divisible, }1\thicksim _{\pi
}kn, sk\thicksim _{\pi }sk+1\text{ }\forall s=1,...,n-1\}
\end{eqnarray*}
 Let $V$ be the block of $\pi $ which contains the element $1$. Since $\pi $ is
$k$-divisible in order that the size of all the blocks of $\pi$ to be multiple of $k$ the last element of $V$ must be $sk$ for some $s\in \{1,...n\}$. But
if $k\neq n\,$ then $sk$ would not be connected to $sk+1\,\ $ in $\pi$ and
neither in $\sigma$.  This of course means that $\pi \vee \sigma \neq
1_{kn}.$ Therefore $1\thicksim _{\pi }kn$.

\begin{figure}[h]
\begin{center}
\begin{picture}(10,2)
 \put(1,1) {$1$}
% \put(1.2,2.5) {\vector(0,-1){1}}
% \put(1,2.6) {$1$}

 \put(2,1) {$2$}
 \put(3,1) {$\cdots$}
 \put(3.75,-.4) {%$\longleftarrow$ 
$V$ %$\longrightarrow$
}
\put(3.55,-.2) {\vector(-1,0){2}}
\put(4.55,-.2) {\vector(1,0){2}}
 \put(4.5,1) {{$sk$-$1$}}

% \put(7.5,2.6) {$sk$}
 \put(6.5,1) {$sk$}
% \put(6.2,2.5) {\vector(0,-1){1}}

% \put(7.5,2.6) {$sk$+$1$}

 \put(8.5,1) {{$sk$+$1$}}
%\put(8.2,2.5) {\vector(0,-1){1}}
 \put(10.6,1) {$\cdots$}

 \multiput(8,1.5)(0,-0.2){16}{\line(0,-1){0.1}}

 \linethickness{2pt}
 \put(2.4,.8){\lhalf{2}{1.5}}
 \put(8,.8){\rhalf{2}{1.5}}
 %\put(2.2,-1.2) {\line(1,0){1.5}}
 \multiput(2.9,-1.2)(0.4,0){6}{\line(1,0){0.2}}

 \put(9.9,.8){\lhalf{2}{1}}

 \multiput(10.4,-1.20)(0.4,0){4}{\line(1,0){0.2}}

\end{picture}
\end{center}
\end{figure}

 Relabelling the elements in $\{1,\dots ,$ $kn\}$ by a rotation of $k$ does not affect the
properties of $\pi $ being $k$-divisible or $\pi \vee \sigma =1_{kn}$, so the same argument implies that $sk\thicksim _{\pi }sk+1$, $\forall k=1,...,n-1.$

Now, the set $\{\pi \in NC(kn)\mid \pi $ k-divisible, $1\thicksim _{\pi
}kn,sk\thicksim _{\pi }sk+1$ $\forall s=1,...,n-1\}$ is in canonical
bijection with $\{\widetilde{\pi }\in NC((k-1)n)\mid \widetilde{\pi }$ is $
(k-1)$-divisible induced by the identification  $sk\equiv sk+1$, for $s=1,...,n-1$, and $1\equiv kn$.  

And since%
\begin{eqnarray*}
\alpha _{kn}^{(k)} &:&=\kappa_{kn}^{x} \\
\alpha _{n}^{(k)} &=&0\text{ if }k\nmid n
\end{eqnarray*}%
Then $\kappa_{\pi }(x,x,...,x,x)\rightarrow \alpha^{(k-1)} _{\widetilde{\pi }}$. So 
\begin{equation*}
\kappa_{n}(x^{k},x^{k},...,x^{k})=\sum_{\substack{ 
\widetilde{\pi }\in NC(kn)\\
\widetilde{\pi}~k\text{-divisible}
} }
\alpha^{(k-1)} _{\widetilde{\pi }}
=\sum_{\mathbf{\pi }\in NC((k-1)n)}\alpha^{(k-1)} _{\pi }
\end{equation*}%
as desired.
\end{proof}

\begin{prop}[Free cumulants of $x^{k}$, Second formula]\label{2formula}
Let $(\mathcal{A},\phi )$ be a non-commutative probability
space and let $x$ a $k$-divisible element with $k$-determining sequence $(\alpha _{n})_{n\geq 1}$. Then the following formula holds for the cumulants
of $x^{k}$.
\begin{equation*}
\kappa_{n}(x^{k},x^{k},...,x^{k})=\sum\limits_{\mathbf{\pi }\in
NC(n)}\beta _{\pi }
\end{equation*}%
where 
\begin{equation}
\beta _{k}=\sum_{\pi \in NC((k-1))n}\alpha _{ \pi}\text{.}
\label{beta}
\end{equation}

\end{prop}

The last theorem gives a moment cumulant formula between $\beta _{n}$ and $\kappa_{n}(x^{k}...,x^{k})$ which for example says that when $\beta _{n}$ is a cumulant sequence then $x^{k+1}$ is a free compound Poisson and then $\boxplus$- infinitely divisible, this will be explained in detail in Section \ref{limits}.
 
\begin{prop}[Free cumulants of $x^{k}$, Third formula]\label{3formula}
\label{kumulants of x^k 3 }Let $(\mathcal{A},\phi )$ be a non-commutative probability
space and let $x$ a $k$-divisible element with $k$-determining sequence $(\alpha _{n})_{n\geq 1}$. Then the following formula holds for the cumulants
of $x^{k}$.
\begin{equation}
\kappa_{n}(x^{k},x^{k},...,x^{k})=[\alpha\ast \underbrace{\zeta \ast \cdots \ast \zeta}_{k~times}]_n.
\end{equation}
\end{prop}

\subsection{Freeness and k-divisible elements} 

Recall the definition of diagonally balanced pairs from Nica and Speicher \cite{NiSp95}.

\begin{defi}
Let $\mathcal{A},\phi$ be a non-commutative probability space, and let $a_1,a_2$ be
in $\mathcal{A}$. We will say that $(a_1,a_2)$ is a diagonally balanced pair if
\begin{equation}
\phi(\underbrace{a_1a_2\cdots a_1a_2a_1}_{2n+1})=\phi(\underbrace{a_2a_1...a_2a_1a_2}_{2n+1})=0
\end{equation}

\end{defi}

Two prominent examples of balanced pairs are $(u,u^*)$ where $u$ is a Haar unitary and $(s,s)$ where $s$ is even. It is well known in free probability that if $a$ is free from $\{u,u^*\}$ then $uau*$ is free from $a$, and similarly if $a$ is free from $s$ then it is also free from $sas$.

More generally, it was proved in \cite{NiSp95} that if $(b_1,b_2)$ is a diagonally balanced pair and $a$ is free from $\{b_1,b_2\}$ then $b_1ab_2$ is free from $a$. Now, notice that if $s$ is $k$-divisible then the pair $(s^i,s^{k-i})$ is diagonally balanced and then ${sas^{k-1},s^2as^{k-2},...,s^{k-1}as}$ and $a$ are free. We can consider instead of $s^ias^{k-i}$  any monomial on $a$ and $s$ of degree $k$ on $s$ and freeness will still hold. Furthermore, we see that if $a$ and $s$ are free and $s$ is $k$-divisible then $shs$ and $a$ are free, where $p$ is any polynomial on $a$ and $s$ of degree $k$ on $s$. This is the content of the next proposition.

\begin{prop}\label{freeness}
 Let $s$ be $k$-divisible and $a$ be free of $s$. And let $h=sa_1sa_2sa_3s...sa_{k-1}s$, where for all $i=1,...,n$ the element $a_i$ is free from $s$. Then $h$ and $a$ are free.
\end{prop}
\begin{proof}
Consider a mixed cumulant of $h$ and $a$ and use the formula for cumulants with products as arguments. 
\begin{equation}
\kappa_n(...,h,...,a,...)=\sum_{\substack{ 
\pi \in NC(n)\\
\pi \vee \sigma =1_{n}
} }
=\kappa_n(...,\underbrace{s,a_1,s,a_2,s,...s,a_{k-1},s}_h,...,a,...)
\end{equation}
Let us analyze the summand of the RHS and show that the must vanish. In order to satisfy the minimum condition $a$ must be joined with some element on $h$. Now, for this $h=sa_1s,a_2s...sa_{k-1}s$, a can not be joined with some $s$, since they are free. So it must join with some $a_i$ as follows.

\begin{figure}[h]
\begin{center}
\begin{picture}(10,3)
\linethickness{2pt}
\put(-5,1,1) {$\kappa_n(...,\overbrace{s,a_1,s,...s,a_i,s,...s,a_{k-1},s}^{h},...,s,a,...)$}
\put(-5,1.1) {$\kappa_n(...,s,a_1,s,...s,a_i,\underbrace{s,...s,a_{k-1},s,...,s}_{km-i~number~of~s},a,...)$}
\put(5.1,1){\lhalf{2}{4}}
\put(13,1){\rhalf{2}{4}}

\put(2.1,1){\rhalf{2}{1}}
\put(10,1){\lhalf{2}{1}}
\multiput(10.5,-1)(0.4,0){4}{\line(1,0){0.2}}
\multiput(-1.6,-1)(0.4,0){5}{\line(1,0){0.2}}
%\put(3,1){\lhalf{1}{3}}
%\put(6,1){\rhalf{1}{3}}
\end{picture}

\end{center}
\end{figure}

 In this case there must be a block of size not multiple of $k$ containing only $s$\'{}s (since $s$ is free from $\{a,a_1,...a_n\}$) and then $\kappa_n(...,s,a_1,s,a_2,s,...s,a_{k-1},s,...,a,...)$ must vanish for all summands in RHS. So any mixed cumulant of $h$ and $a$ vanishes and hence $a$ and $h$ are free.
\end{proof}

\section{R-diagonality} \label{k tuples}

We may generalize the concept of diagonally balanced pair to $k$-tuples.
\subsection {Diagonally balanced k-tuples}

\begin{defi}
Let $\mathcal{A},\phi$ be a non-commutative probability space, and let $a_1,\dots,a_k$ be
in $\mathcal{A}$. We will say that $(a_1,\dots ,a_k)$ is a \emph{diagonally balanced $k$-tuple} if every ordered sequence of 
size not multiple of $k$ vanishes with $\phi$, i.e.
\begin{equation}
\phi(a_ja_{j+1}\cdots a_ka_1\cdots a_ka_1\cdots a_{i-1}a_i)=0
\end{equation}
whenever $a_{j-1}\neq a_i$ (the indices are taken modulo $k$).
\end{defi}

The proof of Proposition \ref{freeness} can be easily modified for diagonally balanced $k$-tuples, and is left to the reader. So we have a more general result.

\begin{thm}
 Let $(\mathcal{A},\phi)$ be a non-commutative probability space, and let $(s_1,\dots ,s_k)$ be
a diagonally balanced $k$-tuple free from $a$. And let $h=s_1a_2s_2a_3s_3\cdots s_{k-1}a_{k-1}s_k$, where for all $i=1,...,n$ the element $a_i$ is free from $\{s_1,\dots ,s_k\}.$ Then $h$ and $a$ are free.
\end{thm}

A special kind of diagonally balanced pair which is very important in the literature of free probability is the one of $R$-diagonal pair, introduced in \cite{NiSp95}. %These $R$-diagonal pairs have given many examples 
There is a lot of structure in these elements and relation to even elements is well known \cite{NiSp06}. Moreover a big class of invariant subspaces have been studied by Speicher and Sniady \cite{S-S} and relation to $R$-cyclic matrices was pointed out in \cite{NSS}. 

\begin{defi}
Let $(\mathcal{A},\phi)$ be a non-commutative probability space, and let $a_1,\cdots,a_k$ be
in $\mathcal{A}$. We will say that $(a_1,\cdots ,a_k)$ is an $R$-diagonal $k$-tuple if the only non-vanishing free cumulants have increasing order, i.e. they are of the form
\begin{equation*}
\kappa_{kn}(a_1,a_2,\dots a_k,a_1,a_2,\dots ,a_k\cdots a_1,a_2,\dots a_k)=\kappa_{kn}(a_ia_{i+1}...a_ka_1\dots a_ka_1\dots a_{k-i+1}).
\end{equation*}

\end{defi}

\begin{rem}
The case $k=2$ was well studied in \cite{NiSp95}. An element a is $R$-diagonal if and only if the pair $(a,a^*)$ is $R$-diagonal. 
\end{rem}

\begin{thm}[cumulants of $R$-diagonal tuples] \label{Rdiag1}
Let $(a_1,\dots,a_s)$ be an $R$-diagonal $k$-tuple in a tracial state and denote by 
\begin{equation}
\alpha_n:=\kappa_{kn}(a_1,\dots,a_k,\dots,a_1\dots a_k)
\end{equation}
Then, if $a=a_1a_2...a_k$, we have
\begin{equation}
\kappa_n(a,\dots,a)=\sum_{\pi\in NC(n)}\alpha^{(k-1)}_\pi
\end{equation}
\end{thm}
\begin{proof}

Again, the formula for products as arguments yields
\begin{equation*}
\kappa_{n}(a,a,...,a)=\sum\limits_{\substack{ \mathbf{\pi }\in
NC(kn)  \\ \pi \vee \sigma =1_{kn}}}\kappa_{\mathbf{\pi }}(a_1,a_2,...,a_{k-1},a_k)
\end{equation*}%
with $\sigma =\{(1,2,3,..k),(k+1,k+2,...,2k),...,(k(n-1)+1,...,nk)\}$.

Observe that by the fact that $(a_1,\dots,a_k)$ is an $R$-diagonal $k$-tuple
\begin{equation*}
\sum\limits_{\substack{ \mathbf{\pi }\in NC(kn)  \\ \pi \vee \sigma =1_{kn} 
}}\kappa_{\mathbf{\pi }}(a_1,a_2,...,a_{k-1},a_k)=\sum\limits_{\substack{ \mathbf{\pi }\in
NC(kn),~\pi ~k\text{-divisible}  \\ \pi \vee \sigma =1_{kn}}}\kappa_{\pi
}(a_1,a_2,...,a_{k-1},a_k)
\end{equation*}%

From this point the argument is identical as in the second proof of Theorem \ref{1formula}.

\end{proof}

Similar formulas as in Theorems \ref{2formula} and \ref{3formula} hold for $R$-diagonal tuples. 

\begin{prop} \label{Rdiag2}
Let $(a_1,\dots,a_k)$ be an $R$-diagonal $k$-tuple in a tracial state and denote by 
\begin{equation}
\alpha_n:=\kappa_{kn}(a_1,\cdots,a_k,\cdots,a_1\cdots a_k)
\end{equation}
The following formulas hold for the cumulants of $a=a_1a_2....a_s$
\begin{equation*} 
\kappa_n(a,\cdots,a)=[\alpha\ast \underbrace{\zeta \ast \cdots \ast \zeta}_{k~times}]_n
\end{equation*}
and
\begin{equation*} 
\kappa_n(a,\cdots,a)=\sum\limits_{\mathbf{\pi }\in
NC(n)}\beta _{\pi }
\end{equation*}%
where 
\begin{equation*}
\beta _{k}=\sum_{\pi \in NC((k-1)n)}\alpha^{(k-1)} _{\pi}\text{.}
\end{equation*}%
\end{prop}

\begin{rem} 
(1) Theorem \ref{Rdiag1} and Proposition \ref{Rdiag2} are also true for diagonally balanced. One can easily modify the proofs by using Remark \ref{rem1}.

(2)Notice that the determining sequence of a diagonally balanced $k$-tuple is determined by the moments of $a=a_1a_2\cdots a_d$ but the same is not true for the whole distribution of $(a_1,a_2,\dots,a_d)$.
\end{rem}

\subsection{R-cyclic matrices}

Let $(\mathcal{A},\phi)$ be a non-commutative probability space, and let $d$ be a positive integer. Consider the algebra $M_d(\mathcal{A})$
of $d\times d$ matrices over $\mathcal{A}$ and the linear functional $\phi_d$ on $M_d(\mathcal{A})$ defined by the formula
\begin{equation}
\phi((a_{i,j})^n_{i,j=1})=\frac{1}{d}\sum^d_{i=1}\phi(a_ii)
\end{equation}
Then $(M_d(\mathcal{A}),\phi_d)$ is itself a non-commutative probability space.

\begin{defi}
Let $(\mathcal{A},\phi)$ and let $A\in(M_n(\mathcal{A}),\phi_n)$, $A$ is said to be $R$-cyclic if the following conditions holds
\begin{equation}
\kappa_n(a_{i_1,j_1},\cdots a_{i_n,j_n})=0
\end{equation}
for every $n>0$ and every $1\leq i_1,j_1,...\leq d$ for which it is not true that $j_1=i_2,\dots,j_n-1=i_n,j_n=i_1$.

\end{defi}

We can realize $k$-divisible elements as $R$-cyclic matrices with $R$-diagonal $k$-tuples as entries.
A formula for the distribution of an $R$-cyclic matrix in terms of its entries was given in \cite{NSS}.  However, in the case treated here, this formula will not be needed in full generality and we will rather use the special information we know to obtain the desired distribution.

\begin{prop} \label{r-cyclic}
Let $(a_1,a_2,...a_k)$ be a tracial diagonally balanced $k$-tuple in a $(\mathcal{A},\phi)$ and consider the superdiagonal matrix
\[ A := \left( \begin{array}{ccccc}
0       &   a_1 & 0   & \cdots & 0       \\
 \vdots &\ddots & a_2 & \ddots & \vdots  \\
        &       &     & \ddots & 0       \\
 0      &       &     & 0      & a_{k-1} \\
a_k     &0      &     & \cdots & 0       \end{array} \right)\]
as an element in $(M_k(\mathcal{A}),\phi_k)$.

(1)$A$ is $k$-divisible.

(2)$A^k$ has the same moments as $a:=a_1\cdots a_k$. In particular, if $a$ is positive $A$ has moments as a $k$-symmetric distribution.

(3)$A$ has the same determining sequence as $(a_1,a_2,...a_k)$.

(4)$A$ is $R$-cyclic if and only if $(a_1,\dots,a_d)$ is an $R$-diagonal tuple.

\end{prop}
\begin{proof}

(1) $A$ is $k$-divisible since the powers if $A$ which are not multiple of $k$ have zero entries in the diagonal.

(2) This is clear since 
\[ A^k := \left( \begin{array}{cccc}
a_1\cdots a_k  &                  &            &  0       \\
               & a_2\cdots a_ka_1 &            &  \\
               &                  & \ddots     &       \\
  0            &                  &            &a_k\cdots a_{k-1}  \end{array} \right)\] 

which by traciality has moments $\phi((a_1...a_k)^n)=\phi(a^n)$.

(3) By Theorems \ref{kumulants of x^k } and \ref{Rdiag1}, the determining sequence of $A$ depends on the moments of $A^k$ in the same way as $(a_1,a_2,...,a_k)$ so by $(2)$ the determining sequences 
must coincide.

(4) The definition of $R$-ciclicality says that $\kappa_{n}(a_{i_1},a_{1_2},\dots a_{i_n}=0$ whenever is not true that $i_2=i_1+1$, $i_3=i_2+1$,...,$i_1=i_{n}+1$. This is equivalent to the fact that $n$ is multiple of $k$ and the indices are increasing, which is exactly the definition of $R$-diagonal tuples.
\end{proof}

\begin{exa}[free $k$-Haar unitaries] \label{freehk}

The simplest example of the last theorem is given by taking $a_i=1$. \[ A := \left( \begin{array}{ccccc}
0       &   1 & 0   & \cdots & 0       \\
 \vdots &\ddots & 1 & \ddots & \vdots  \\
        &       &     & \ddots & 0       \\
 0      &       &     & 0      & 1 \\
1     &0      &     & \cdots & 0       \end{array} \right)\]
Clearly, this matrix is $k$-Haar unitary, with distribution $\mu_A=\frac{1}{k}\sum_{j=1}^k\delta_{q^j}$
as an element in $(M_k(\mathcal{A}),\phi_k)$. Notice that, instead of the upperdiagonal matrix, we can choose any permutation matrix of size $nk\times nk$ in which any cycle has length $k$. Of course, if we choose at random one of them, we still get a $k$-Haar Unitary.
Moreover, Neagu \cite{Neagu} proved that if we let $N\rightarrow \infty$ we get asymptotic freeness in the following sense. 
\begin{thm}
Let $\{U_1^N,U_2^N,...,U_r^N\}_{N>0}$ be a family of independent random $Nk\times Nk$ permutation matrices with cycle lengths of size $k$, then as $N$ goes to infinity $U_1^N,U_2^N,...,U_r^N$ converges in $*$-distribution to a $*$-free family $u_1,u_2,...,u_r$ of non-commutative random variables with each $u_i$ $k$-Haar unitary. \end{thm}
This gives a matrix model for $u_1,...,u_r$ free $k$-Haar unitaries. 
\end{exa}

\section{Main Theorem and first consequences}\label{mainsection}

 In this, the main section of the paper, we will prove the Main Theorem. This theorem will not only allow us to define free multiplicative convolution between $k$-symmetric distributions and probability measures in $\mathcal{M}^+$ but, moreover, will permit us to define free additive convolution powers for $k$-symmetric distributions. Also, in the combinatorial side, we generalize Theorem \ref{MainRemark} to any multiplicative family.

The main tool that we will use is the $S$-transform. This $S$-transform has not been defined for $k$-divisible random variables, the principal problem is on choosing an inverse for the transform $\psi$.

\subsection{The $S$-transform for random variables with $k$ vanishing moments}

We will start in the very general setting of an algebraic non-commutative probability space $(\mathcal{A},\phi)$ and define an $S$-transform for random variables such the first $k-1$ moments equal $0$.

Recall the definition of the $S$-transform for positive measures. For a probability measure $\mu$ on $\real$, we let
$\psi_\mu(z):=\int_{\real}\frac{zx}{1-zx}\mu(dx)$. $\psi_\mu$ coincides with a moment generating function if $\mu$ has finite moments of all orders. The $S$-transform is defined as
\begin{equation}\label{eq4}
S_\mu(z):= \frac{1+z}{z}\chi_\mu(z),~~~z \in  \psi_\mu(i\comp_+).
\end{equation}

In general, when $x$ is a selfadjoint random variable with non-vanishing mean the $S$-transform can be defined as follows.

\begin{defi}
Let $x$ be a random variable with $\phi(x)\neq 0$. Then its $S$-transform is defined as follows. Let $\chi$ denote the inverse under composition of the series 
\begin{equation}
\psi(z):=\sum^\infty _{n=1}\phi(x^n)z^n,
\end{equation}
then 
\begin{equation}
S_x(z):=\chi(z)\frac{1+z}{z}.
\end{equation}
\end{defi}

Here, $\phi(x)\neq0$ ensures that the inverse of $\psi$ exists as a formal power series. The importance of the $S$-transform is the fact that $S_{xy}=S_xS_y$ whenever $\phi(x)\neq 0$ and $\phi(y)\neq 0$. 

We want to consider the case when $\phi(x)=0$. The case when $x$ is selfadjoint and $\phi(x^2)>0$  was treated in Raj Rao and Speicher in \cite{RS}. The main observation is that although $\psi$ cannot be inverted by a power series in $z$ it can be inverted by a power series in $\sqrt{z}$. This inverse is not unique, but there are exactly two choices. 

The more general case where $\phi(x^n)=0$ for $n=1,2,...,k-1$ and $\phi(x^k)\neq0$ can be treated in a similar fashion. In this case there are $k$ possible choices to invert the function $\psi$. We include the proof for the convenience of the reader.

\begin{prop}
 Let $\psi(z)$ be a formal power series of the form
\begin{equation}
\psi(z)=\sum^\infty _{n=k}\alpha_nz^n
\end{equation}
with $\alpha_k>0$. There exist exactly $k$ power series in $z^{1/k}$ which satisfy
\begin{equation}
\psi(\chi(z))=z.
\end{equation}
\end{prop}
\begin{proof}
Let 
\begin{equation}
\chi(z)=\sum^\infty _{i=1}\beta_iz^{i/k}
\end{equation}
The equation $\psi(\chi(z))=z$ is equivalent to
 \begin{equation}
\sum^\infty _{n=k}\alpha_n(\sum^\infty _{n=1}\beta_iz^{i/k})^n=z.
\end{equation}
This yields to the system of equations  
\begin{equation*}
1=\alpha_k\beta^k_1
\end{equation*}
and
\begin{equation*}
0=\sum^r _{n=k}\sum^r _{i_1+\dots+i_n=r}\alpha_n\beta_{i_n}\dots\beta_{i_n}
\end{equation*}
for all $r>2$. Clearly the solutions of the first equation are
\begin{equation*}
\beta_1=\alpha^{1/k}_k
\end{equation*} 
while the other equations ensure that $\beta_n$ is determined by $\beta_1$ and the $\alpha$'s. 
\end{proof}

Now, we can define the $S$-transform for random variables having vanishing moments up to order $k-1$.

\begin{defi}\label{S1} Let $x$ be a random variable with $\phi(x^n)=0$ for $n=1,2,\dots,k-1$ and $\phi(x^k)>0$. Then its $S$-transform is defined as follows. Let $\chi(z)=\sum^\infty _{i=1}\beta_iz^{i/k}$ be the inverse in under composition of the series 
\begin{equation}
\psi(z)=\sum^\infty _{n=k}\phi(x^n)z^n
\end{equation}
with leading coefficient $\beta_1>0$. Then 
\begin{equation}
S_x(z)=\chi(z)\frac{1+z}{z}.
\end{equation}

\end{defi}

The following theorem is a generalization of Theorem $2.5$ in \cite{RS} and shows the role of the $S$-transform with respect to multiplication of free random variables.

\begin{thm}\label {Stran2}
Let $x\in (\mathcal{A},\phi)$ such that $\phi(x^n)=0$ for $n=1,2,...,k-1$ and $\phi(x^k)>0$ and let $y\in (\mathcal{A},\phi)$ such that $\phi(y)\neq0$. If $S_x$ and $S_y$ denote their respective $S$-transforms, then 
\begin{equation*}
S_{xy}(z)=S_xS_y(z),
\end{equation*}
where $S_{xy}$ is the $S$-transform of $xy$. 
\end{thm}
\begin{proof}

The proof is exactly the same as in \cite{RS}. The only observation to be made is that $xy$ also satisfies the conditions in Definition \ref{S1}. Indeed, by freeness $\phi((xy)^n)=0$ for $n=1,2,...,k-1$ and $\phi((xy)^k)=\phi(x^k)\phi(y)^k>0$ and then all the manipulations are valid for the case when $k>2$. The key point
is to verify that $C_{xy}(S_{xy}(z))=z$.
\end{proof}

\begin{rem} We cannot drop the assumption $\phi(y)\neq0$ in Theorem \ref{Stran2}. As pointed out by Rao and Speicher \cite{RS}, freeness would yield to $\phi((yx)^n)=0$, for all $n\in\mathbb{N}$.
\end {rem}

\subsection{Free Multiplicative convolution of $k$-symmetric distributions}

Recall the notion of free multiplicative convolution of two measures $\mu$ in $\mathcal{M}$ and $\nu$ in $\mathcal{M}^+$. The idea is to consider a positive free random variables $x$ and a selfeadjoint random variable $y$ (free from $x$) with distributions $\mu$ and $\nu$, respectively, and call $\mu\boxtimes\nu$ the distribution of $x^{1/2}yx^{1/2}$. This element is selfadjoint so we can be sure that $\mu\boxtimes\nu$ is a well defined probability measure on $\mathcal{M}$, but moreover $x^{1/2}yx^{1/2}$ and $xy$ have the same moments. In other words, $\mu\boxtimes\nu$ can be defined as the only distribution in $\mathcal{M}$ whose moments equal the moments of $xy$.

Following these ideas, the strategy is clear in how to define a free multiplicative convolution $\mu\boxtimes\nu$ for $\mu$ $k$-symmetric and $\nu$ with positive support. We consider a $k$-divisible random variable $x$ and a positive element $y$ (free from $x$) with distributions $\mu$ and $\nu$, respectively. Given a $k$-divisible random variable $x$ and a positive one it is clear that $xy$ is a also $k$-divisible in the algebraic sense. The interesting question is to find an element with $k$-symmetric distribution with the same moments as $xy$. In this section we prove that this element does exist. Observe that in this case taking the random variable $x^{1/2}yx^{1/2}$ does not work since it is not necessarily normal.
 
Recall that given a $k$-symmetric probability measure $\mu$ on $\mathbb{R}$, we denote by $\mu^{k}$ the probability measure in $\mathcal{M}^{+}$ induced by the map $t\rightarrow t^{k}$. 
%In other words if $X$ is a $k$-divisible element with distribution $\mu$, then $\mu^{k}$ is the distribution of $X^k$.

 We start by stating a relation between the $S$-transform of a $k$-divisible element $x$ and the $S$-transform of $x^k$.
\begin{lem}
Let $x\in(\mathcal{A},\phi)$ be a $k$-divisible element. Then the $S$-transforms of $x$ and $x^k$ are related by the formula 
\begin{equation*}
S_{x^k}(z)=S_{x}(z)^k(\frac{z}{1+z})^{k-1}.
\end{equation*}
\end{lem}

\begin{proof}
By definition $m_n(x^k)=m_{nk}(x)$ and $m_{s}(x)=0$ if $k\nmid s$. So
 \begin{equation*}
\psi_x(z)=\sum^\infty _{n=1}m_{n}(x)=\sum^\infty _{n=1}m_{nk}(x)z^nk
\end{equation*}
and 
 \begin{equation*}
\psi_{x^k}(z)=\sum^\infty _{n=1}m_{n}(x^k)=\sum^\infty _{n=1}m_{nk}(x)z^n
\end{equation*}
Thus $\psi_x(z)=\psi_{x^k}(z^k)$, or equivalently,  $\chi_{x^k}(z)=\chi_x(z)^k$
and then
\begin{equation*}
S_{x}(z)^k=(\frac{1+z}{z})^k\chi_x(z)^k=(\frac{1+z}{z})^k\chi_{x^k}(z)=(\frac{1+z}{z})^{k-1}S_{x^k}(z).
\end{equation*}
So
\begin{equation*}
S_{x^k}(z)=S_{x}(z)^k(\frac{z}{1+z})^{k-1} 
\end{equation*}

\end{proof}

Now we are in position to prove the  Main Theorem.
\begin{mthm}%[Main Theorem] 
\label{MT}
 Let $x,y\in(\mathcal{A},\phi)$ with $x$ positive and $y$ a $k$-divisible element. Consider $x_1,...,x_k$ positive elements with the same moments as $x$. Then $(xy)^k$ and $y^kx_1\cdots x_k$ have the same moments, i.e.
\begin{equation}
 \phi((xy)^{kn})=\phi((y^kx_1\cdots x_k)^n)
\end{equation}
\end{mthm}
\begin{proof}
It is enough to check that the $S$- transforms of $(xy)^k$ and $y^kx_1\cdots x_k$ coincide. Now
\begin{eqnarray*}
S_{(xy)^k}(z)&=&S_{xy}(z)^k(\frac{z}{1+z})^{k-1}=S_{x}(z)^kS_{y}(z)^k(\frac{z}{1+z})^{k-1}\\
&=&S_{x}(z)^k\cdot S_{y^k}(z)=S_{(x_1)}(z)\cdots S_{x_k}(z)\cdot S_{y^k}(z)\\
&=&S_{y^kx_1\cdots x_k}(z)
\end{eqnarray*}

\end{proof}

\begin{rem}
In the tracial case, Theorem \ref{freeness} gives another proof of Main Theorem. Indeed, consider the moments of $sas...sasa$ when $s$ is $k$-divisible, since $sas...sas$, and $a$ are free, by Theorem \ref{freeness}, then these moments coincide with the moments of $sas..sasa_1$ where $a_1$ is free from $s$ and $a$. Now by traciality the moments of $sas...sas$ coincide with the moments of $s^2as...sa$ which again, Theorem \ref{freeness} coincide with the moments of $s^2as...sa_2$ where $a_2$ is free from $s$ and $a$. So the moments of $sas...sasa$ coincide with the moments of $s^2as...sa_2a_1$ with $a_1,a_2,a$ and $s$ free between them. Continuing with this procedure we see that the moments of $sas...sasa=(sa)^k$ coincide with the moments of $s^ka_1a_2\cdots a_k$, with $a_i$'s and $s$ free between them.

\end{rem}

The next corollary allows us to define free multiplicative convolution between $k$-symmetric and probability measures in $\nu\in\mathcal{M}^+$.

\begin{cor}
Let $x$ be $k$-divisible with $x^k$ positive and let $y$ be positive. For $Z=(xy)^k$ there is a positive element $\hat Z$ with $\phi(Z^n)=\phi(\hat Z^n)$
\end{cor}

\begin{defi}
Let $\mu\in\mathcal{M}^+$ and let $\nu\in\mathcal{M}^k$ be a $k$-symmetric probability measure. And suppose that $\mu$ and $\nu$ are the distributions of $X$ and $Y$, free elements in some probability space $(\mathcal{A},\phi)$, respectively. We define $\mu\boxtimes\nu=\nu\boxtimes\mu$ to be the unique $k$-symmetric probability measure with the same moments as $XY$.
\end{defi}

\begin{rem}
Notice that the last definition does not depend on the choice of $X$ and $Y$ since the distribution of $X$ and $Y$ (by freeness) determine the moments of $XY$ moments uniquely.
\end{rem}

Finally we obtain the mentioned relation.
\begin{cor}\label{cormulconv}
Let $\mu\in\mathcal{M}^+$ and let $\nu\in\mathcal{M}^k$. The following formula holds:
\begin{equation}\label{multconv}
(\mu\boxtimes\nu)^k=\mu^{\boxtimes k}\boxtimes\nu^{k}.
\end{equation}

\end{cor}

\begin{rem}\label{k-diagonal}
One may ask if any measure $k$-divisible measure can be decomposed as the free multiplicative convolution of a $k$-Haar $ \nu_k=\frac{1}{k}\sum_{j=1}^k\delta_{q^j}k$ and a positive measure. However, Corollary \ref{cormulconv} shows that this is not the case since $$(\mu\boxtimes\nu_k)^k=\mu^{\boxtimes k}.$$
\end{rem}

\subsection{Free additive powers}

 Just as in the multiplicative case, it is not straightforward that free additive convolution for $k$-symmetric distributions is well defined. In fact, at this point this is an open problem. 

\textbf{Open Question.} Can we define free additive convolution of $k$-symmetric probability measures?

We will give a partial answer in the next section, see Theorem \ref{freeconvK}.
However, another important consequence of the Main Theorem is the existence of free additive powers of $\mu$, when $\mu$ is a probability measure with $k$-symmetry.

\begin{thm}
 Let $\mu\in\mathcal{M}_k$ be a $k$-symmetric distribution. Then for each $t>1$ there exists a $k$-symmetric measure $\mu^{\boxplus t}$ with $\kappa_n( \mu^{\boxplus t})= tk_n(\mu)$. 
\end{thm}

\begin{proof}
Let $x\in(\mathcal{A},\phi)$ be a tracial $C$*-probability space and let  $x\in(\mathcal{A},\phi)$ be such that $x^k$ is positive and with distribution $\mu$ and $p\in(\mathcal{A},\phi)$ a projection such that $\phi(p)=\frac{1}{t}$, with $x$ and $p$ free. Now consider the compressed space $(p\mathcal{A}p,\phi^{p\mathcal{A}p})$ and the element $x_t:=pXp\in(p\mathcal{A}p,\phi^{p\mathcal{A}p})$, with $X=tx$. By Theorem 14.10 in \cite{NiSp06} the cumulants of $x_t$ (with respect to $\phi^{p\mathcal{A}p}$) are 
\begin{equation}
\kappa^{p\mathcal{A}p}_n(X_t,\dots,X_t)=t\kappa_n(\frac{1}{t}tX,\dots,\frac{1}{t}tX)=t\kappa_n(X,\dots ,X)
\end{equation}
Now, X is $k$-divisible and $X^k$ is positive  and $p$ is positive so, by the Main Theorem, the moments of $Xp$ also define $k$-symmetric distribution. Also, since $\phi$ is tracial we have
\begin{equation}
\phi^{p\mathcal{A}p}(pXppXp\cdots pXp)=\phi(pXpXp\cdots pXpXp)=\phi(XpX\cdots Xp)
\end{equation}
this means that the moments of $(pXp)^k$ define a positive measure $\mu$. Now consider the compressed $(pXp)^k$. Then
\begin{equation}
\phi^{p\mathcal{A}p}(pXppXp\cdots pXp)=\frac{1}{t}\phi(pXpXp\cdots pXpXp)=\phi(XpX\cdots Xp)
\end{equation}
but the measure $\nu=(1-1/t)\delta_0+1/t\mu)$ has moments $m_n(\nu)=\frac{1}{t}m_n(\mu)$ and we are done.
\end{proof}
 
Although we are not able to define free additive convolution for all $k$-symmetric measures, having free additive powers is enough to talk about central limit theorems and Poisson type ones. This will be done in Section \ref{limits}.

\subsection{Combinatorial consequences}

The following theorem of Nica and Speicher \cite{NS96} gives a formula for the moments and free cumulants of product of free random variables. 
\begin{thm}  \label{momentofprod}
Let $(A,\phi )$ be a non-commutative probability space and consider the free random
variables $a, b\in A$. The we have 
\begin{equation*}
\phi((ab)^n)=\sum_{\pi \in NC(n)} \kappa_{\pi
}(a)\phi _{K(\pi )}(b^n)
\end{equation*}%
and 
\begin{equation*}
\kappa_{n}(a)=\sum_{\pi \in NC(n)} \kappa_{\pi
}(a)\kappa_{K(\pi )}(b)
\end{equation*}
\end{thm}

The observation here is that we can go the other way. Indeed for two multiplicative family ${f_n}$ and ${g_n}$ we can find a probability space $(A,\phi)$, and elements $a$ and $b$ in $A$ such that 
$\kappa_n^a=f_n$ and $\phi(b^n)=g_n$ and then we can calculate $(f*g)_n$ by the formula $(f*g)_n=\phi((ab)^n)$. Using this idea and the Main Theorem  we can generalize broadly Theorem \ref{MainRemark} to any multiplicative family whose first element is not zero.

\begin{thm}
\label{MainRemark2} \ The following statements are equivalent.

1) The sequence $f_n$ is given by the $k$-fold convolution
\begin{equation*}
f_{n}=g _{n}\ast\underbrace{h _{n}\ast \cdots \ast h _{n}}_{k~times} 
\end{equation*}
2) The dilated sequence $f_{n}^{(k)}$ is given by the convolution
\begin{equation*}
f_{n}^{(k)}=g_{n}^{(k)}\ast h _{n}
\end{equation*}

\end{thm}

\begin{proof}
In the proof of the Main Theorem, from the combinatorial point of view, positivity is not important. So let $X,Y$ be in $(A,\phi)$ with $Y$ a $k$-divisible element, and assume that $X$ has cumulants
 $\kappa_n^x=h_n$ and $Y^k$ has moments $\phi((Y^k)^n)=g_n$ (and therefore $\phi(X^n)=g_n^{(k)}$). If we consider $X_1,...,X_k$ elements with the same moments as $X$. Then $(XY)^k$ and $X_1\cdots X_kY^k$ have the same moments, i.e.
\begin{equation}\label{master}
 \phi((XY)^{kn})=\phi((X_1\cdots X_kY^k)^n)
\end{equation}
Now, the moments of $X_1\cdots X_kY^k$ are given by
\begin{equation*}
f_n:=\phi((Y^kX_1\cdots X_k)^n)=g_n\ast h_n\ast\cdots *h_n 
\end{equation*}
and the moments of $XY$ are given by
\begin{equation*}
\tilde f_n:=\phi((XY)^{n})=g_n^{(k)} \ast h_{n}
\end{equation*}
Now Equation (\ref{master}) implies that $\tilde f_n=f_{n}^{(k)}$. 
\end{proof}

\section{Limit theorems and free infinite divisibility} \label{limits}

In this section we will address questions regarding limit theorems. First, we prove central limit theorems and compound type ones, for $k$-symmetric measures. Next,  we consider the free infinite divisibility. Finally, we study the free multiplicative convolution of measures on the positive real line from the point of view of $k$-divisible partitions and its connections to the free multiplicative convolution between $k$-symmetric measures.

\subsection{Free central limit theorem for $k$-divisible measures}
 We have a new free central limit theorem for $k$-symmetric measures. Recall that for a measure $\mu$, $D_t(\mu)$ denotes the dilation by $t$ of the measure $\mu$.

\begin{thm}[Free Central limit theorem for $k$-symmetric measures] \label{freeCLT}

Let $\mu$ be a $k$-symmetric measure with finite moments and $\kappa_k(\mu)=1$ then, as $N$ goes to infinity,
\begin{equation*}
D_{N^{-1/k}}(\mu^{\boxplus N})\rightarrow s_k,
\end{equation*}
where $s_k$ is the only $k$-symmetric measure with 
free cumulant sequence $\kappa_n(s_k)=0$ for all $n\neq k$ and $\kappa_k(s_k)=1$. Moreover,
 $$(s_k)^k=\pi^{\boxtimes k-1},$$
 where $\pi$ is a free Poisson measure with parameter 1.
\end{thm}
\begin{proof}
Convergence in distribution to a measure determined by moments is equivalent to the convergence of the free cumulants. 
Now, for $i=1,2,...,n-1$ the $i$-th free cumulant $\kappa_i(\mu^{\boxplus N})$ equals zero and for $i> k$, the $i$-th free cumulant
	$$
\kappa_i(D_{N^{-1/k}}(\mu^{\boxplus N}))=(N^{-1/k})^i\kappa_i(\mu^{\boxplus N})=\frac{N}{N^{i/k}}\kappa_i(\mu)={N^{1-i/k}}\kappa_i(\mu)\rightarrow0$$
when $N$ goes to infinity. So, in the limit, the only non vanishing free cumulant is $\kappa_k(D_{N^{-1/k}}(\mu^{\boxplus N}))=\kappa_k(\mu)=1$. This means that $ s_k$ is the only $k$-symmetric measure with 
free cumulant sequence $\kappa_n=0$ for all $n\neq k$ and $\kappa_k=1$. For the second statement, on one hand, we calculate the moments of $s_k$ using the moment cumulant formula:
\begin{eqnarray}
 m_n(s_k^k) =m_{nk}(s_k)&=&\sum_{\pi\in NC(nk)} \kappa_\pi(s_k)  \\
                   &=&\sum_{\pi\in NC_k(n)} 1  \\
&=&\frac{\binom{kn}{n}}{kn-1}.
\end{eqnarray}
On the other hand, the moments of $\pi^{\boxtimes k-1}$ are known to be (See \cite{BBCC} or Example \ref{freebesselex} below).
$$m_n(\pi^{\boxtimes k-1})=\frac{\binom{kn}{n}}{kn-1}.$$
\end{proof}

\begin{rem}
We can derive properties of from the fact that $(s_k)^k=\pi^{\boxtimes k-1}$. Indeed, let $B(0,r)=\{z\in \mathbb{C} : |z|<r\}$. The measure $s_k$ satisfies the following properties.

(i) There are no atoms.

(ii) The support is $B(0,K)\cap \mathcal{A}_k$, where $K=\sqrt[k]{(k)^{k}/(k-1)^{k-1}}$.

(iii)The density is analytic on $(0,K)$.
%(iv) The moments are given by $m_{nk}=\frac{\binom{kn}{n}}{kn-1}$.
\end{rem}

\begin{rem}
(1) Note from the proof of Theorem \ref{freeCLT} that in the algebraic sense we only need the first $k-1$ moments to vanish. For $k=1$, this is the law of large numbers and for $k=2$ we obtain the usual free central limit theorem.

(2) Observe that $s_k$ satisfies a stability condition. Indeed, $$s_k^{\boxtimes 2}=D_{2^{1/k}}(s_k)$$ from where we can interpret $s_k$ as a strictly stable distribution of index $k$. This raises the question whether there are other $k$-symmetric stable distributions. Of course, in the presence of moments we can only get a $s_k$ from the free central limit theorem above. Hence, if we expect to find other stable distribution we need to extend the notion of free additive powers to $k$-symmetric measures without moments. This will be done in Section \ref{Analytic Aspects}.
 
(3) The law of small numbers and more generally free compound Poisson type limit theorems are also valid for $k$-symmetric distributions. Moreover, a notion of free infinite divisibility will be given and studied. This is the content of next parts of this section.

\end{rem}

\subsection{Compound free Poissons}

The analogue of compound Poisson distributions and infinite divisibility is are the subjects of this subsection. Recall the definition of a free compound Poisson on $\real$.
\begin{defi}
A probability measure $\mu$ is said to be a free compound Poisson of rate $\lambda$ and jump distribution $\nu$ if the free cumulants $(\kappa_n)_{n\geq1}$ of $\mu$ are given by $\kappa_n(\mu)=\lambda m_n(\nu) $. In this case, $\lambda \nu$ coincides with the L\'evy measure of $\mu$.
\end{defi}

The most important free compound Poisson measure is the Marchenko-Pastur law $\pi$ whose $R$-transform is $R_\pi(z)=\frac{z}{1-z}$.
$\pi$ is also characterized by $S_\pi(z)=\frac{1}{z+1}$ in terms of the $S$-transform.

Following the definition of a free compound Poisson for selfadjoint random variables we can define their analogues for $k$-symmetric distributions. 

\begin{defi}
 A $k$-symmetric distribution $\mu$ is called a free compound Poisson of rate $\lambda$ and jump distribution $\nu$ if the free cumulants $(\kappa_n)_{n\geq1}$ of $\mu$ are given by $\kappa_n(\mu)=\lambda m_n(\nu) $, for some $\nu$ a  $k$-symmetric distribution.
\end{defi}

The existence of these measures can be easily proved by finding explicitly $\pi(\lambda,\nu)^k$.%, this will be done in Theorem ?.$
 As announced we have a limit theorem for the free compound Poisson distributions. We shall mention that, implicitly, Banica et al. \cite{BBCC} observed the case $\nu=\frac{1}{k}\sum_{j=1}^k\delta_{q^j}$
\begin{thm}
We have the Poisson type limit convergence
\begin{equation*}
((1-\frac{\lambda}{N})\delta_0+\frac{\lambda}{N}\nu)^{\boxplus N}\rightarrow \pi(\lambda,\nu)
\end{equation*}
\end{thm}
\begin{proof}
The proof is identical as for the selfadjoint case, see for example \cite{NiSp06}. The main observation is that if $\nu_N=((1-\frac{\lambda}{N})\delta_0+\frac{\lambda}{N}\nu)^{\boxplus N}$ then 
\begin{equation*}
\kappa_n(\nu_N)=\frac{\lambda}{N} m_n(\nu)+O(1/N^2)
\end{equation*}
and then $\kappa_n(\nu^{\boxplus N}_N)=N\kappa_n(\nu_N)$ converges to $\lambda m_n(\nu)$.
\end{proof}

\begin{exa}[Free Bessel laws] \label{freebesselex}

Free Bessel laws  introduced in \cite{BBCC}, are defined by
\begin{equation*}
\pi_{kt}=\pi^{\boxtimes k}\boxtimes\pi^{\boxplus k}.
\end{equation*}
We restrict attention to the case $t = 1$, for simplicity. 
They proved using a matrix model that the free Bessel law $\pi_{k1}$ with $k\in\nat$ is given by
\begin{equation}\label{bes1}
\pi_{k1}=law\left[\sum_{j=1}^k\ \left[ P_j{q^j}\right]\right]^k
\end{equation} 
where $P_1,...,P_k$ s are free random variables, each of them following the free Poisson
law of parameter $1/k$.
So they were lead to consider the modified free Bessel laws $\hat\pi_{s1}$, given by
\begin{equation}\label{bes2}
\hat\pi_{s1}=law\left[\sum_{j=1}^s\ \left[ P_j{q^j}\right]\right].
\end{equation} 
It is important to notice that $\sum_{j=1}^k\ \left[ W_j{q^j}\right]$ is not a normal operator so the equalities in (\ref{bes1}) and (\ref{bes2}) are just equalities in moments (and not $*$-moments). In our notation means that
\begin{equation*}
\hat\pi_{k1}=\hat\pi_{k1}^k
\end{equation*}

A modified free Bessel law is $k$-symmetric, but moreover it is a compound free Poisson with rate $\lambda=1$ and jump distribution a $k$-Haar measure. So we have the representation
\begin{equation*}
\hat\pi_{k1}=\pi(1,\nu)=\pi\boxtimes\frac{1}{k}\sum_{j=1}^k\delta_{q^j}
\end{equation*}
Combining these identities we see that
\begin{equation*}
(\pi\boxtimes\frac{1}{k}\sum_{j=1}^k\delta_{q^j})^k=(\pi(1,\nu))^k=\hat\pi_{k1}=\pi_{k1}^k=\pi^{\boxtimes k}.
\end{equation*}
which is nothing but Equation (\ref{multconv}) for $\mu=\pi$ and $\nu=\sum_{j=1}^k\delta_{q^j}$. Moreover the free cumulants and moments of $\pi^{\boxtimes k}$ are given by 
$$m_n( \pi^{\boxtimes k})=\frac{\binom{(k+1)n}{n}}{kn+1}~~~k_n(\pi^{\boxtimes k})=\frac{\binom{kn}{n}}{{(k-1)n+1}}.$$
This is easily seen since the free cumulants of $\pi$ are given by $k_n(\pi)=1$ for all $n\in\mathbb{N}$. So calculating the moments and cumulants of $\pi^{\boxtimes k}$ amounts counting the number of $k$-multichains of $NC(n)$ which was done in Example \ref{k-k-k}.
\end{exa}

\subsection{Free infinite divisibility}
 Given the limit theorems above, the concept of free infinite divisibility  in $mathcal{M}_k$ raises naturally.
\begin{defi}
A $k$-symmetric measure is $\boxplus$-infinitely divisible if for any $N>0$ there exist $\mu_N$ such that $\mu_N^{\boxplus N}=\mu$. We will denote the set of freely infinitely divisible distribution in $\mathcal{M}_k$ by $ID^\boxplus(\mathcal{M}_k)$
\end{defi}

It is easily seen the $ID^\boxplus(\mathcal{M}_k)$ is closed under convergence in distribution.
Free compound Poissons are $\boxplus$-infinitely divisible, since $\pi(\lambda,\mu)^{\boxplus t}=\pi(\lambda t,\mu)$. Moreover any free infinitely divisible measure can be approximated by free compound Poissons. The proof of this fact follows the same lines as for the selfadjoint case. We will give the main ideas of this proof for the convenience of the reader. 

The following is a special case of Lemma 13.2 in Nica Speicher \cite{NiSp06}.
\begin{lem}\label{lemma71}
Let $a_N$ be random variables in some non-commutative probability space $(\mathcal{A},\phi_N)$ and denote then the following statements are equivalent.

(1)For each $n\geq1$ the limit 
\begin{equation*}
 \lim_{N\rightarrow \infty}N\cdot \phi_N(a^n_N)
\end{equation*}
exists.

(2)For each $n\geq1$ the limit
\begin{equation*}
 \lim_{N\rightarrow \infty}N\cdot \kappa^N_n(a_N,...,a_N)
\end{equation*}
exists.
Furthermore the corresponding limits are the same.
\end{lem}

Now, we can prove the approximation result. 
\begin{prop}\label{approx}
A $k$-symmetric measure with is freely infinitely divisible if and only if it can be approximated (in distribution) by free compound Poissons.
\end{prop}
\begin{proof}
On one hand, since free compound Poissons are freely infinitely divisible any measure approximated by them is also infinitely divisible. On the other hand, let $\mu$ be infinitely divisible. Then for any $N>0$ there exist $\mu_N$ such that $\mu_N^{\boxplus N}=\mu$. So by Lemma \ref{lemma71} we have
\begin{equation}
\kappa_n(\mu)=N\cdot \kappa_n(\mu_N)=\lim_{N\rightarrow \infty}N\cdot \kappa_n(\mu_N)=\lim_{N\rightarrow \infty}N\cdot m_n(\mu_N)
\end{equation}
Now, let $\nu_N$ be a free compound Poisson with rate $N$ and jump distribution $\mu_N$ then $\kappa_n(\nu_N)=Nm_n(\mu_N)$
\begin{equation}
\lim_{N\rightarrow \infty}\kappa_n(\nu_N)=\lim_{N\rightarrow \infty}N\cdot m_n(\mu_N)
\end{equation}
So $\nu_N\rightarrow\mu$ in distribution.
\end{proof}

Next, the results of Section 5  can be interpreted in terms of free compound Poissons.

\begin{prop}
\label{Poisson mult} Suppose that $x$ is a $k$-divisible element and $
\alpha _{n}=\kappa_{kn}(x)$ is a free cumulant sequence of a positive element $(\kappa_{n}(a)=\alpha _{n})$ with distribution $\nu $ then
\begin{equation*}
distr(x^{k})=\pi ^{\boxtimes k-1}\boxtimes \nu.
\end{equation*}
\end{prop}
\begin{proof}

By Proposition \ref{3formula} we have that the free cumulants of $x^k$ are given by
\begin{equation*}
\kappa_n(x^k)=[\alpha *\zeta \cdots *\zeta]_n.
\end{equation*}
On the other hand, by successive application of Equation (\ref{cum-prod}), we can see that the cumulants of $\pi ^{\boxtimes (k-1)}\boxtimes \nu$ are given by
\begin{equation*}
\kappa_n(\pi ^{\boxtimes (k-1)}\boxtimes \nu)=[\alpha *\zeta \cdots *\zeta]_n
\end{equation*}
as desired.
\end{proof}

\begin{cor}
If $x$ is a $k$-symmetric compound Poisson with rate $\lambda $ and jump
distribution $\nu.$ Then the distribution of $x^{k}$ is a compound Poisson with rate $1$ and jump
distribution $\pi^{\boxtimes k-1} \boxtimes \nu ^{k}$ .
\end{cor}

\begin{proof}
If $x$ is $k$-symmetric compound Poisson with levy measure $\mu $ then $
\kappa_{n}(x)=m_{n}(\nu ).$ So, $\alpha _{n}=\kappa_{kn}(x)=m_{kn}(\nu )=m_{n}(\nu
^{k})$, that is $\alpha _{n}$ is the free cumulant sequence of $\pi \boxtimes \nu^k$. By the last proposition 
\begin{equation*}
distr(x^{k})=\pi ^{\boxtimes k}\boxtimes \nu^k .
\end{equation*}%
In other words $x^k$ is a free compound Poisson with levy measure $\pi^{\boxtimes k-1}\boxtimes
\nu^k .$
\end{proof}

We prove that free infinite divisibility is maintained under the mapping $\mu\rightarrow\mu^k$, this generalizes results of \cite{AHS} where the case $k=2$ was considered.

\begin{thm} \label{freeinf10}
If $\mu$ is $k$-symmetric and $\boxplus$-infinitely divisible, then $\mu ^{k}$ is
also $\boxplus$-infinitely divisible.
\end{thm}

\begin{proof}
Suppose that $\mu $ is infinitely divisible. Then $\mu $ can be approximated
by free compound Poisson which are $k$-symmetric. Say $\mu =\lim_{n\rightarrow \infty
}\mu _{n}$ where $\mu _{n}=\pi \boxtimes \nu _{n}$. By the last corollary
there $\mu _{n}^{k}=\pi^{ \boxtimes k} \boxtimes \nu _{n}.$ Now $\mu _{n}^{k}\rightarrow
\mu ^{k}$ and since $ID^{\boxplus }(\mathcal{M}_k)$ is closed in the weak convergence topology we have that $\mu $ is infinitely divisible.
\end{proof}

\begin{cor}
A $k$-symmetric infinitely divisible measure has at most $k$-atoms. 
\end{cor}
\begin{proof}This follows from the well known result of Bercovici and Voiculescu \cite{Be-Vo} that a freely infinitely divisible measure on the real line has at most $1$ atom.
\end{proof}

Finally we come back to the question of defining free convolution. We give a partial answer to the question raised in last section.

\begin{thm} \label{freeconvK} Let $\mu$ and $\nu$ be $k$-symmetric freely infinitely divisible measures. Then there exists a $k$-symmetric $\mu\boxplus\nu$ such that $$\kappa_n(\mu\boxplus\nu)=\kappa_n(\mu)+\kappa_n(\nu).$$ Moreover $\mu\boxplus\nu$ is also freely infinitely divisible.
\end{thm}
\begin{proof}
Since the free convolution of $k$-divisible free compound Poisson is also a $k$-divisible free compound the by Theorem \ref{approx} this is also true $k$-symmetric freely infinitely divisible measures.
\end{proof}

It would be interesting to give a Levy-Kintchine Formula and study triangular arrays for $k$-symmetric probability measures.

\subsection{Free multiplicative powers of measures on $\real^+$ revisited}
In this section, for a probability measure $\mu\in\mathcal{M}_+$ with compact support we will denote by $\mu^{1/k}$ the positive measure with $m_{nk}(\mu^{1/k})=m_n(\mu)$ and $\mu^{[1/k]}$ the $k$-symmetric measure such that $m_{nk}(\mu^{[1/k]})=m_n(\mu)$.
Consider Remark \ref{k-diagonal} for $ \nu=\frac{1}{k}\sum_{j=1}^k\delta_{q^j}$, a $k$-Haar measure. Then $$(\mu\boxtimes\sum_{j=1}^k\delta_{q^j})^k=\mu^{\boxtimes k}.$$ Using this fact, the moments of $\mu^{\boxtimes k}$ may be calculated using $k$-divisible non-crossing partitions as we show in the following proposition.
\begin{thm} 
Let $\mu$ be a measure with positive support. Then the moments of $\mu_k:=\mu^{\boxtimes k}$ are given by
\begin{eqnarray}
m_n(\mu_k)&=& \sum_{\pi \in NC^k(n)}\kappa_{Kr(\pi)}(\mu), \label{kmom2} 
\end{eqnarray}
where $NC^k(n)$ denotes the $k$-divisible partitions of $[kn]$.
\end{thm}

\begin{proof}
Let $\nu=\sum_{j=1}^k\delta_{q^j}$, the moments of $\mu\boxtimes\nu$ can be calculated using Theorem \ref{momentofprod}:
$$m_n(\mu\boxtimes\nu))=\sum_{\pi \in NC^k(n)}\kappa_{Kr(\pi)}(\mu)m_\pi(\nu)=\sum_{\pi \in NC^k(n)}\kappa_{Kr(\pi)}(\mu)$$
where the last equality follows since $m_\pi(\nu)=0$ unless $\pi$ is $k$-divisible.
\end{proof}

This formula has been generalized for non-identically distributed random variables in \cite{A-V} where it was used to give new proofs of results in Kargin \cite{Kar,Kar2} and Sakuma and Yoshida \cite{SaYo} regarding the asymptotic behaviors of $\mu^{\boxtimes k}$ and $(\mu^{\boxtimes k})^{\boxplus k}$, respectively. 

Moreover, from results of Tucci \cite{Tu} we know that the $k$-th root of the measure $\mu^{\boxtimes k}$ converges to a non-trivial measure. More precisely, he proved the following.

\begin{thm}\label{tu2}
Let $\mu$ be a probability measure with compact support. If we denote by $\mu_k=(\mu^{\boxtimes k})^{1/k}$
then $\mu_k$ converges weakly to $\hat\mu$, where $\hat\mu$ is the unique measure characterized by $ \hat\mu( [0,\frac{1}{S_\mu(t-1)} ])=t $ for all $t\in(0,1)$. The support of the measure $ \hat \mu$ is the closure of the interval
 $$(\alpha,\beta)=((\int_0^\infty x^{-1}d\mu(x))^{-1},\int_0^\infty xd \mu (x)),$$
where $0\leq\alpha<\beta\leq\infty$
\end{thm}

On the other hand, for $R$-diagonal operators, Haagerup and Larsen \cite{HaL} proved the following. 
\begin{thm} Let $T$ be an $R$-diagonal operator and $t\in(0,1)$. If $\nu:=\mu_{|T|^2}$ is not a Dirac measure then $\mu_T( B(0,\frac{1}{\sqrt{S_\nu(t-1)}}))=t$
where $B(0,r)=\{z\in \mathbb{C} : |z|<r\}$
\end{thm}

If we combine these two results we obtain the following interesting interpretation of the limiting distribution. 

\begin{thm}\label{polar} Let $a,u\in A$ be free elements with a positive and $u$ a Haar unitary. Moreover, let $\mu$ be a probability measure with compact support distributed as $a^2$. If we denote by 
\begin{equation}\label{polarr}
\mu_k=\mu\boxtimes\frac{1}{k}\sum_{j=1}^k\delta_{q^j}
\end{equation}
then $\mu_k$ converges weakly to $\mu_\infty$ where $\mu_\infty$ is the rotationally invariant measure such that $\mu_\infty(B(0,t^2))=\mu_{au}(B(0,t))$, where $\nu$ is the Brown measure of $au$.
\end{thm}

\begin{proof}
Let $T=au$ and then $|T|^2=a^2$, so $\mu_{|T|^2}=\mu$. Now,
since $(\mu^{\boxtimes k})^{1/k}$ converges to $\hat\mu$, then $\mu\boxtimes\sum_{j=1}^k\delta_{q^j}=(\mu^{\boxtimes k})^{[1/k]}$ converges to the rotationally invariant measure $\mu_\infty$ with $\mu_\infty (B(0,t))=\hat\mu(0,t)$. This implies that $$\mu_\infty (B(0,\frac{1}{S_\mu(t-1)}))=t=\mu_T( B(0,\frac{1}{\sqrt{S_\mu(t-1)}}))$$ and then $\mu_\infty (B(0,t^2))=\mu_{au}(B(0,t))$, as desired.
\end{proof}

\begin{rem} (1)Haagerup and M\"oller \cite{HaM} have generalized results of \cite{Tu} to unbounded operators. The previous theorem can be generalized to unbounded operators using the analytic methods of next section.

(2) Recall from Example \ref{freehk} that random permutation matrices with cycles of size $k$ are asymptotically free $k$-Haar unitaries. One can think of a Haar unitary as a limit of $k$-Haar unitaries, from previous theorem $R$-diagonal elements can be thought as the limit of $k$-divisible ones of the type (\ref{polarr}). 
\end{rem}

\begin{exa}[$\infty$-semicircle]
Let $s_k$ be the $k$-semicircle distribution from Theorem \ref{freeCLT}. Then there exist a measure $s_\infty$ such that $$\lim_{k\to\infty} s_k\to s_\infty.$$ Combining Theorems  \ref{freeCLT}, and \ref{polar} one can see that $s_\infty(B(0,t))=t$. 

Indeed, since $(s_{k+1})^{k+1}=\pi^{\boxtimes k}=(\pi\boxtimes\frac{1}{k}\sum_{j=1}^k\delta_{q^j})^k$ then by Theorem \ref{polar}, $s_\infty(B(0,t^2))=\mu_{au}(B(0,t))=t^2$, where $a$ is a quarter circular (see Example 5.2 of \cite{HaL}).

\end{exa}

\section{The unbounded case} \label{Analytic Aspects}

We end by generalizing some of our results to $k$-symmetric measures without moments.
The \textit{\ free multiplicative convolution} $\boxtimes$ for general measures in
$\mathcal{M}^+$ was defined in \cite{BeVo93} using operators affiliated to a $W*$-algebra. 
This convolution is characterized by $S$-transforms defined as follows. 
For a general probability measure $\mu$ on $\mathbb{R}$, let
\begin{equation}
\Psi_{\mu}(z)=\int_{\mathbb{R}}\frac{zt}{1-zt}\mu(\mathrm{d}t)=\frac{1}%
{z}G_{\mu}\left(  \frac{1}{z}\right)  -1,\quad z\in\mathbb{C}\backslash
\mathbb{R}_{+}. \label{PhsiCT}%
\end{equation}
The function $\Psi_{\mu}$ determines the measure $\mu$ uniquely since the
Cauchy transform $G_{\mu}$ does. $\Psi_\mu$ coincides with a moment generating function if $\mu$ has finite moments of all orders. The next result was proved in \cite{BeVo93} for probability measures in
$\mathcal{M}^{+}$ with unbounded support.
\begin{prop}
\label{PisiForPos} Let $\mu\in\mathcal{M}^{+}$ such that $\mu(\{0\})<1$. The
function
\begin{equation}
\Psi_{\mu}(z)=\int_{0}^{\infty}\frac{zx}{1-zx}\mu(\mathrm{d}x),\quad
z\in\mathbb{C}\backslash\mathbb{R}_{+} \label{DefPsi}%
\end{equation}
in univalent in the left-plane $i\mathbb{C}_{+}$ and $\Psi_{\mu}%
(i\mathbb{C}_{+})$ is a region contained in the circle with diameter
$(\mu(\{0\})-1,0)$. Moreover, $\Psi_{\mu}(i\mathbb{C}_{+})\cap\mathbb{R}%
=(\mu(\{0\})-1,0)$.
\end{prop}

Let $\chi_{\mu}:$ $\Psi_{\mu}(i\mathbb{C}_{+})$ $\rightarrow i\mathbb{C}_{+}$
be the inverse function of $\Psi_{\mu}.$ The $S$-\textit{transform} of $\mu$
is the function
\[
S_{\mu}(z)=\chi(z)\frac{1+z}{z}\text{.}%
\]

\begin{prop}[\cite{BeVo93}]
\label{StransfPos2} Let $\mu_{1}$ and $\mu_{2}$ be probability measures in
$\mathcal{M}^{+}$ with $\mu_{i}\neq\delta_{0}$, $i=1,2.$ Then $\mu
_{1}\boxtimes$ $\mu_{2}\neq\delta_{0}$ and
\[
S_{\mu_{1}\boxtimes\mu_{2}}(z)=S_{\mu_{1}}(z)S_{\mu_{2}}(z)
\]
in that component of the common domain which contains $(-\varepsilon,0)$ for
small $\varepsilon>0.$ Moreover, $(\mu_{1}\boxtimes$ $\mu_{2})(\{0\})=\max
\{\mu_{1}(\{0\}),\mu_{2}(\{0\})\}.$
\end{prop}

Free multiplicative convolution $\mu_1 \boxtimes \mu_2$ can be defined for any two probability measures $\mu_1$ and $\mu_2$ on $\mathbb{R}$, provided that one of them is supported on $[0,\infty)$. However, it is not known whether an $S$-transform can be defined for every probability measure.
However, Arizmendi and P\'erez-Abreu \cite{APA} defined an $S$-transform of a symmetric probability measures. 

We will define the free multiplicative convolution between measures $\mu\in\mathcal{M}_k$ and $\nu\in\mathcal{M}^+$. We generalize the $S$-transform to $k$-symmetric measures; we follow similar strategies as in \cite{APA} and show the multiplicative property still holds for this $S$-transform. 

\subsection{Analytic aspects of $S$-transforms}

Recall that for a $k$-symmetric probability measure $\mu$ on $\mathbb{R}$, let $\mu^{k}$ be the probability measure in $\mathcal{M}^{+}$ induced by the map
$t\rightarrow t^{k}$. 

We define the Cauchy transform a $k$-symmetric distribution $\mu$ by the formula
\begin{eqnarray*} 
G_{\mu}(z)=\int_{\comp}\frac{1}{z-t}\mu(\mathrm{d}t)
\end{eqnarray*}
and the $\Psi$ function in a similar way as (\ref{PhsiCT})
\begin{eqnarray} \label{PhsiCTk}
\Psi_{\mu}(z)=\int_{\comp}\frac{zt}{1-zt}\mu(\mathrm{d}t)=\frac{1}
{z}G_{\mu}\left(  \frac{1}{z}\right)  -1,\quad z\in\mathbb{C}\backslash
\mathbb{R}_{+}. 
\end{eqnarray}

The following two important relations between the Cauchy transforms and the
$\Psi$ functions of $\mu$ and $\mu^{k}$ were proved in \cite{APA} for $k=2$. The proof presented here is the same with obvious changes; we present it for the convenience of the reader.

\begin{prop}
\label{CTmuCTmu2} Let $\mu$ be a $k$-symmetric probability measure $\mu$ on
$\mathbb{R}$. Then

a) $G_{\mu}(z)=z^{k-1}G_{\mu^{k}}(z^{k}),$ $z\in\mathbb{C}\backslash\mathbb{R}_{+}.$

b) $\Psi_{\mu}(k)=\Psi_{\mu^{k}}(z^{k}),$ $z\in\mathbb{C}\backslash
\mathbb{R}_{+}.$
\end{prop}

\begin{proof}
a) Use the $k$-symmetry of $\mu$ twice to obtain
\begin{align*}
G_{\mu}(z)  &  =\int_{\mathbb{R}}\frac{1}{z-t}\mu(\mathrm{d}t)=\sum^\kappa_{i=1}\int
_{\mathbb{R}_{+}}\frac{1}{z-t\omega_i}\mu(\mathrm{d}t)\\
&  =kz^{k-1}\int_{\mathbb{R}_{+}}\frac{1}{z^{k}-t^{k}}\mu(\mathrm{d}t)=z\int
_{supp(\mu)}\frac{1}{z^{k}-t^{k}}\mu(\mathrm{d}t)\\
&  =z^{k-1}\int_{\mathbb{R}_{+}}\frac{1}{z^{k}-t}\mu^{k}(\mathrm{d}t)=zG_{\mu^{k}%
}(z^{2}).
\end{align*}

b) Use (\ref{PhsiCTk}) twice and (a) to obtain
\begin{align*}
\Psi_{\mu}(z)  &  =\frac{1}{z}G_{\mu}\left( \frac{1}{z}\right)  -1\\
&  =\frac{1}{z^{k}}G_{\mu^{k}}\left(  \frac{1}{z^{k}}\right)  -1=\Psi_{\mu
^{k}}(z^{k})
\end{align*}
which shows (b).
\end{proof}

An important consequence is that the function $G_{\mu}$ determines the measure $\mu$ uniquely since the Cauchy transform $G_{\mu^k}$ determines $\mu^k$ and then $\mu$. Also, the function $\Psi_{\mu}$ determines
the measure $\mu$ uniquely since the Cauchy transform $G_{\mu}$ does.

\begin{thm}
\label{main0}Let $\mu$ be a $k$-symmetric probability measure $\mu$ on
$\mathbb{R}$.

a) If $\mu\neq\delta_{0}$, the function $\Psi_{\mu}$ is univalent on $\mathbb{H}_k:=\{z\in \comp_+:\arg z\in {\pi/2k}<\arg z<{3\pi/2k} \}$.
Therefore $\Psi_{\mu}$ has a unique inverse on $H$, $\chi_{\mu}:\Psi_{\mu
}(H)\rightarrow H$.

b) If $\mu\neq\delta_{0}$, the $S$-transform%
\begin{equation}
S_{\mu}(z)=\frac{1+z}{z}\chi_{\mu}(z)
\end{equation}
satisfies
\begin{equation} \label{S-transform 2}
S_{\mu}^{k}(z)=(\frac{1+z}{z})^{k-1} S_{\mu^{k}}(z)
\end{equation}
for $z$ in $\Psi_{\mu}(H_k)$.
\end{thm}

\begin{proof}
a) On one hand, let $T:\mathbb{C\rightarrow C}$ be the function $T(z)=z^{k}.$ Then
$T(\mathbb{H}_k)=i\mathbb{C}_{+}$ and therefore $h$ is univalent in $\mathbb{H}_k$.On the other hand, since $\mu^{k}%
\in\mathcal{M}^{+}$, by Proposition \ref{PisiForPos}, $\Psi_{\mu^{k}}(z)$ is
univalent in $i\mathbb{C}_{+}$ and therefore $\Psi_{\mu^{k}}(z^{k})$ is
univalent in $\mathbb{H}_k$.

b) Since $\mu^{k}\in\mathcal{M}^{+}$, from Proposition \ref{PisiForPos}, the
unique inverse $\chi_{\mu^{k}}$ of $\Psi_{\mu^{k}}$ is such that $\chi
_{\mu^{k}}:$ $\Psi_{\mu^{k}}(i\mathbb{C}_{+})$ $\rightarrow i\mathbb{C}_{+}.$
Thus, use (a) to obtain $\Psi_{\mu^{k}}(\chi_{\mu}^{k}(z))=\Psi_{\mu}%
(\chi_{\mu}(z))=z$ for $z\in\Psi_{\mu}(\mathbb{H}_k)$ and the uniqueness of $\chi
_{\mu^{k}}$ gives $\chi_{\mu^{2}}(z)=\chi_{\mu}^{2}(z)$, $z\in\Psi_{\mu}(\mathbb{H}_k).$
Hence
\begin{align*}
S_{\mu}^{2}(z)  &  =\chi_{\mu}^{2}(z)(\frac{1+z}{z})^{2}=\chi_{\mu^{2}%
}(z)(\frac{1+z}{z})^{2}\\
&  =S_{\mu^{2}}(z)\frac{1+z}{z},\quad z\in\Psi_{\mu}(H).
\end{align*}
as desired
\end{proof}

\subsection{Free multiplicative convolution}

Now, we are in position to define free multiplicative convolution for measures with unbounded support. We will use Equation (\ref{multconv}) as our definition. The definition using free operators on a $W^*$-algebra will be addressed in a forthcoming paper.
\begin{defi}
 Let $\mu$ be $k$-symmetric and $\nu$ be a measure in $\mathcal{M}^+$. The free multiplicative convolution between $\mu$ and $\nu$ is defined unique $k$-symmetric measure $\mu\boxtimes\nu$  such that 
$$ (\mu\boxtimes\nu)^k=\mu^k\boxtimes\nu^{\boxtimes k} $$ 
\end{defi}

\begin{rem}
The fact that the last definition makes sense is justified as follows: $\mu^k$ is  and $\nu^{\boxtimes k}$ are in  $\mathcal{M}^+$  and then $\mu^k\boxtimes\nu^{\boxtimes k}$ also belongs to  $\mathcal{M}^+$. So the symmetric pull back under $x^k$  of the measure is unique and well defined.
\end{rem}

Now we show how to compute free multiplicative convolution of a $k$-symmetric probability measure and a probability measure supported on $[0,\infty)$. No existence of moments or bounded supports for the measures assumed.
\begin{thm}
 Let $\mu$ be $k$-symmetric and $\nu$ be a measure in $\mathcal{M}^+$ with respective S transforms $S_\mu (z)$ and $S_\nu(z)$ then
$$S_{\mu\boxtimes\nu}(z)= S_\mu (z)S_\nu(z) $$ 
\end{thm}
\begin{proof}
From Equation (\ref{S-transform 2})
\begin{eqnarray*}
S_{\mu\boxtimes\nu}^{k}(z)&=&(\frac{1+z}{z})^{k-1}S_{(\mu\boxtimes\nu)^{k}}(z)
=(\frac{1+z}{z})^{k-1} S_{\mu^k\boxtimes\nu^{\boxtimes k}}(z)\\
&=&(\frac{1+z}{z})^{k-1} (\frac{1+z}{z})^{k-1}S_{\mu^k}(z)S_{\nu}^k=
S_{\mu}^{k}(z)S_{\nu}^k.
\end{eqnarray*}
\end{proof}

\begin{rem}
By standard approximation arguments all the the theorems regarding freely infinite divisibility are valid for the unbounded case. 
\end{rem}

\subsection {Stable distributions}

Now we come back to the question of stability. A real probability measure  $\sigma _{\alpha }$ is said to be $\boxplus$- stable of \emph{index} $\alpha$ if $\sigma _{\alpha }^{\boxplus 2}\boxplus\delta_t=D_{2^{1/\alpha}}(\sigma _{\alpha })$ for some $t$. If $t=0$, we say that $\sigma _{\alpha }$ is $\boxplus$-strictly stable. 
Note that, among $k$-symmetric stable measures we can only have strictly stable laws since adding non-trivial Dirac measure is not closed in $\mathcal{M}_k$.

Closely related to the notion of stability is that of domains of attraction. Recall that for a probability measure $\mu$ we say that $\nu$ is in the \emph{free domain of attraction} of $\nu$ if there exists $\alpha$ such that $D_{N^{\alpha}}(\mu^{\boxplus N})\rightarrow \nu_i$. The following theorem explains the relation between domains of attraction and stable laws.

\begin{thm}
Assume that $\mu\in\mathcal{M}$ is not a point mass. Then $\nu$ is
$\boxplus$-stable if and only if the free domain of attraction of $\nu$ is not empty.
\end{thm}

 As we have mentioned before, $s_k$ is strictly stable of index $k$. We begin by showing that for each $k$ and each $\alpha\in(0,k]$ there is a $k$-symmetric strictly stable law of index $\alpha$ (that we will denote $\sigma_{k,\alpha}$). In fact, we have an explicit representation of $\sigma_{k,\alpha}$ as the free multiplicative convolution between a $k$-semicircular distribution and strictly stable distribution on $\real$.
This result was proved in \cite{APA} for symmetric distributions in real line and in \cite{Be-Pa} for positive measures.

\begin{thm}\label{st1}
For $k>0$and $0<\alpha\leq 1$, let $\beta=\frac{ k\alpha}{\alpha+k-k\alpha}$ , then the measure $\sigma^k_\beta:=w_k\boxtimes\nu_\alpha$ is stable of index $\beta$. The $S$-transform of
$\sigma^k_\beta$ is given by 
\begin{equation}
S_\beta=\theta_\beta e^{i(1-\beta)\frac{\pi}{\beta}}
\end{equation}
\end{thm}
\begin{proof}
 The $S$-transform for positive strictly stable laws is found in \cite{APA} and can be easily derived from the appendix in \cite{BePa99}: $$S_\alpha=\theta_\alpha e^{i(1-\alpha)\frac{\pi}{\alpha}}z^{\frac{1-\alpha}{\alpha}}.$$
A direct calculation shows that the $S$-transform of $w_k$ is $$S_{w_k}
=z^{\frac{1-k}{k}}.$$
Thus, the $S$ transform of $w_k\boxtimes\nu_\alpha$ is given by 
$$S_{w_k\boxtimes\nu_\alpha}(z)=\theta_\alpha e^{i(1-\alpha)\frac{\pi}{\alpha}}z^{\frac{1-\alpha}{\alpha}+\frac{1-k}{k}}.$$
Hence, on one hand, from (\ref{S-free}) we get 
\begin{eqnarray}
S_{(w_k\boxtimes\nu_\alpha)^{\boxplus 2}}(z)&=&\frac{1}{2}S_{(w_k\boxtimes\nu_\alpha)}(z/2) \\
&=&1/2\cdot\theta_\alpha e^{i(1-\alpha)\frac{\pi}{\alpha}}(\frac{z}{2})^{\frac{1-\alpha}{\alpha}+\frac{1-k}{k}} \\
&=&1/2^{1/\beta}\cdot\theta_\alpha e^{i(1-\alpha)\frac{\pi}{\alpha}} z^{\frac{1-\alpha}{\alpha}+\frac{1-k}{k}}.
\end{eqnarray}
On the other hand, from (\ref{S-Dil}) we have $$S_{D_{2^{1/\beta}}(w_k\boxtimes\nu_\alpha)}(z)=\frac{1}{2^{1/\beta}}\cdot\theta_\alpha e^{i(1-\alpha)\frac{\pi}{\alpha}} z^{\frac{1-\alpha}{\alpha}+\frac{1-k}{k}}.$$
\end{proof}

\begin{conj}\label{conjstable}
Let $k>2$, the $k$-symmetric measures $\sigma^k_\beta$ defined in Theorem \ref{st1} are the only $k$-symmetric $\boxtimes$-stable distributions.
\end{conj}

The following reproducing property was proved in \cite{Be-Pa} for one sided free stable distributions: \begin{equation} \label{biane2}
\nu_{1/(1+t)}\boxtimes\nu_{1/(1+s)}= \nu_{1/(1+t+s)},
\end{equation}
while for the real symmetric free stable distribution the analog relation was proved in \cite{APA}.
\begin{equation}
\sigma_{1/(1+t)}\boxtimes\nu_{1/(1+s)}= \sigma_{1/(1+t+s)}.
\end{equation} 

A generalization for $k$-symmetric distributions is also true, the proofs in \cite{APA} and \cite{Be-Pa} rely on an explicit calculation of the $S$-transform and can be easily modified to this framework.

\begin{thm}\label{Tstable}
For any $s,r>0$, let $\sigma^k_{1/(1+r)}$ be a $k$-symmetric strictly stable distribution of index $1/(1+r)$ and $\nu_{1/(1+s)}$ be a positive strictly stable distribution of index $1/(1+s)$. Then
\begin{equation}
\sigma^k_{1/(1+t)}\boxtimes\nu_{1/(1+s)}= \sigma^k_{1/(1+t+s)}.                           \end{equation}
\end{thm}
\begin{proof} This follows from Theorem \ref{st1}, indeed letting $\beta=(k-t+kt)/(tk)$
\begin{eqnarray*}
\sigma^k_{1/(1+t)}\boxtimes\nu_{1(1+s)}&=&w_k\boxtimes\nu_{1+\beta}\boxtimes\nu_{1+s}\\
&=&w_k\boxtimes\nu_{(1+\beta+s)}\\
&=&\sigma^k_{1/(1+t+s)}.
\end{eqnarray*}
We used \ref{biane2} in the second inequality.
\end{proof}

We have the following conjecture regarding domains of attraction.
\begin{conj}\label{conjstable2}

Assume that $\mu\in\mathcal{M}_k$ is not a point mass. Then $\nu$ is
$\boxplus$-stable if and only if the free domain of attraction of $\nu$ is not empty.
\end{conj}

Now, Theorem \ref{Tstable} may be explained by the following observation.

\begin{lem}\label{Domains}
Let $\mu_1$ and $\mu_2$ be in the $\boxplus$-domain of attraction of $\nu_1$ and $\mu_2$, respectively. Then $\mu_1\boxtimes \mu_2$ is in the $\boxtimes$-domain of attraction of $\nu_1\boxtimes \nu_2$.
\end{lem}
\begin{proof}
For $i=1,2$, since  $\mu_i\in \mathcal{D}^\boxplus(\nu_i)$ then there are some $\alpha_i$\'{}s such that $D_{N^{\alpha_i}}(\mu^{\boxplus N})\rightarrow \nu_i$. 
Now using Equation (\ref{mult-additive}) we have 
$$(\mu_1\boxtimes\mu_2)^{\boxplus N}=D_N(\mu_1^{\boxplus N}\boxtimes\mu_2^{\boxplus N})$$
and dilating by $N^{\alpha_1+\alpha_2-1}$ we get 
\begin{equation}\label{stable2}
D_{N^{\alpha_1+\alpha_2-1}}((\mu_1\boxtimes\mu_2)^{\boxplus N})
=D_{N^{\alpha_1+\alpha_2}}(\mu_1^{\boxplus N}\boxtimes\nu_2^{\boxplus N})  =D_{N^{\alpha_1}}(\mu_1^{\boxplus N})\boxtimes D_{N^{\alpha_2}}(\mu_2^{\boxplus N}) 
\end{equation}

The RHS of the Equation (\ref{stable2}) tends to $\nu_1\boxtimes\nu_2$, and so the LHS. This of course means that $\mu_1\boxtimes\mu_2\in\mathcal{D}^\boxplus(\nu_1\boxtimes\nu_2)$.
\end{proof}

\begin{rem}
A better look to the proof of Lemma \ref{Domains} gives another proof of the reproducing property for $k=1,2$ and for general $k$ if Conjectures \ref{conjstable} and \ref{conjstable} are true.

Indeed, for any $s,t>0$, let $\sigma^k_{1+t}$ a $k$-symmetric strictly stable distribution of index $1/(1+s)$ and $\nu_{1+s}$ be a positive strictly stable distribution of index $1/(1+t)$.
The measure $\sigma^k_{1+t}\boxtimes\nu_{1+s}$ is clearly $k$-symmetric and strictly stable since  $\mathcal{D}(\sigma^k_{1+t}\boxtimes\nu_{1+s})$ is non-empty by the last lemma. The index of stability can be easily calculated from Equation (\ref{stable2}), since in this  $$D_{N^{1+s+1+t-1}}(\sigma^k_{1+t}\boxtimes\nu_{1+s})\rightarrow\sigma^k_{1+t}\boxtimes\nu_{1+s}$$
which means that $\sigma^k_{1+t}\boxtimes\nu_{1+s}$ is a $k$-symmetric strictly stable distribution of index $1/(1+s+t)$.
\end{rem}

Finally, recall from Theorem \ref{freeinf10} that the $k$-power of a freely infinitely divisible measure in $\mathcal{M}_k$ is also freely infinitely divisible. In the case of stable laws we can identify explicitly the Levy measure, for $k\geq2$. Indeed, since $$(w_k\boxtimes\nu_{1/(1+s)})^k=w_k^k\boxtimes(\nu_{1/(1+s)})^{\boxtimes k}=\pi^{\boxtimes {k-1}}\boxtimes \nu_{1/(1+ks)},$$
the Levy measure is given by $\pi^{\boxtimes k-2}\boxtimes \nu_{1/(1+ks)}$.

\section*{Acknowledgement}

I thank my advisor, Roland Speicher, for many helpful discussions and encouragement. I am also grateful to Professor James Mingo for many discussions during the time I spent at Queen's University.

\end{document}